\documentclass[a4paper, 10pt]{amsart}

\usepackage[utf8]{luainputenc}
\usepackage[T1]{fontenc}
\usepackage{graphicx}
\usepackage[english]{babel}
\usepackage{enumitem}
\usepackage{amsthm}
\usepackage{amsmath}
\usepackage{amssymb}
\usepackage{stmaryrd}
\usepackage{amsfonts}
\usepackage{macros_for_trees}
\usepackage{tikz}
\usepackage{bm}
\usepackage{dsfont}
\usetikzlibrary{arrows}
\usepackage[colorlinks=true, pdfstartview=FitV, linkcolor=blue, citecolor=blue, urlcolor=blue,pagebackref=false]{hyperref}
\usepackage{geometry}
\DeclareMathOperator{\Span}{Span}

\makeatletter
\pgfdeclareshape{crosscircle}
{
  \inheritsavedanchors[from=circle] % this is nearly a circle
  \inheritanchorborder[from=circle]
  \inheritanchor[from=circle]{north}
  \inheritanchor[from=circle]{north west}
  \inheritanchor[from=circle]{north east}
  \inheritanchor[from=circle]{center}
  \inheritanchor[from=circle]{west}
  \inheritanchor[from=circle]{east}
  \inheritanchor[from=circle]{mid}
  \inheritanchor[from=circle]{mid west}
  \inheritanchor[from=circle]{mid east}
  \inheritanchor[from=circle]{base}
  \inheritanchor[from=circle]{base west}
  \inheritanchor[from=circle]{base east}
  \inheritanchor[from=circle]{south}
  \inheritanchor[from=circle]{south west}
  \inheritanchor[from=circle]{south east}
  \inheritbackgroundpath[from=circle]
  \foregroundpath{
    \centerpoint%
    \pgf@xc=\pgf@x%
    \pgf@yc=\pgf@y%
    \pgfutil@tempdima=\radius%
    \pgfmathsetlength{\pgf@xb}{\pgfkeysvalueof{/pgf/outer xsep}}%  
    \pgfmathsetlength{\pgf@yb}{\pgfkeysvalueof{/pgf/outer ysep}}%  
    \ifdim\pgf@xb<\pgf@yb%
      \advance\pgfutil@tempdima by-\pgf@yb%
    \else%
      \advance\pgfutil@tempdima by-\pgf@xb%
    \fi%
    \pgfpathmoveto{\pgfpointadd{\pgfqpoint{\pgf@xc}{\pgf@yc}}{\pgfqpoint{-0.707107\pgfutil@tempdima}{-0.707107\pgfutil@tempdima}}}
    \pgfpathlineto{\pgfpointadd{\pgfqpoint{\pgf@xc}{\pgf@yc}}{\pgfqpoint{0.707107\pgfutil@tempdima}{0.707107\pgfutil@tempdima}}}
    \pgfpathmoveto{\pgfpointadd{\pgfqpoint{\pgf@xc}{\pgf@yc}}{\pgfqpoint{-0.707107\pgfutil@tempdima}{0.707107\pgfutil@tempdima}}}
    \pgfpathlineto{\pgfpointadd{\pgfqpoint{\pgf@xc}{\pgf@yc}}{\pgfqpoint{0.707107\pgfutil@tempdima}{-0.707107\pgfutil@tempdima}}}
  }
}
\makeatother

\definecolor{connection}{rgb}{0.7,0.1,0.1}

\tikzset{
root/.style={circle,fill=black!50,inner sep=0pt, minimum size=3mm},
        dot/.style={circle,fill=black,inner sep=0pt, minimum size=1.2mm},
        sdot/.style={circle,fill=black,inner sep=0pt,minimum size=.5mm},
        dotred/.style={circle,fill=black!50,inner sep=0pt, minimum size=2mm},
        var/.style={circle,fill=black!10,draw=black,inner sep=0pt, minimum size=3mm},
        kernel/.style={semithick,shorten >=2pt,shorten <=2pt},
        kernel1/.style={draw=black,thick},
        kernels/.style={snake=zigzag,shorten >=2pt,shorten <=2pt,segment amplitude=1pt,segment length=4pt,line before snake=2pt,line after snake=5pt,},
        rho/.style={densely dashed,semithick,shorten >=2pt,shorten <=2pt},
           testfcn/.style={dotted,semithick,shorten >=2pt,shorten <=2pt},
           tau/.style={circle,inner sep=1pt,draw=black,fill=white,text=black,thin},
        renorm/.style={shape=circle,fill=white,inner sep=1pt},
        labl/.style={shape=rectangle,fill=white,inner sep=1pt},
        xic/.style={very thin,circle,fill=symbols,draw=black,inner sep=0pt,minimum size=1.2mm},
        xi/.style={very thin,circle,fill=blue!10,draw=black,inner sep=0pt,minimum size=1.2mm},
        xix/.style={crosscircle,fill=blue!10,draw=black,inner sep=0pt,minimum size=1.2mm},
	xib/.style={very thin,circle,fill=blue!10,draw=black,inner sep=0pt,minimum size=1.6mm},
	xie/.style={very thin,circle,fill=green!50!black,draw=black,inner sep=0pt,minimum size=1.6mm},
	xid/.style={very thin,circle,fill=symbols,draw=black,inner sep=0pt,minimum size=1.6mm},
	xibx/.style={crosscircle,fill=blue!10,draw=black,inner sep=0pt,minimum size=1.6mm},
	kernels2/.style={very thick,draw=connection,segment length=12pt},
	not/.style={thin,circle,fill=symbols,draw=connection,fill=connection,inner sep=0pt,minimum size=0.5mm},
	>=stealth,
  }

\numberwithin{equation}{section}
\newtheorem{theorem}{Theorem}[section]
\newtheorem{corrolaire}[theorem]{Corollary}
\newtheorem{proposition}[theorem]{Proposition}
\newtheorem{definition}[theorem]{Definition}
\newtheorem{remark}[theorem]{Remark}
\newtheorem{lemme}[theorem]{Lemma}

\newtheorem{assumption}[theorem]{Assumption}

\newcommand{\scal}[2]{\langle #1 , #2 \rangle}
\newcommand{\bscal}[2]{\bm{\langle \langle }  #1, #2 \bm{\rangle \rangle}}
\newcommand{\Bscal}[2]{\bm{ \Big\langle \Big\langle }  #1, #2 \bm{\Big\rangle \Big \rangle}}
\newcommand{\E}{\mathbb{E}}

\newcommand{\ind}{\mathds{1}}

\renewcommand*{\epsilon}{\varepsilon}

\newcommand{\TT}{\TT}
\newcommand{\CF}{\mathcal{F}}

\newcommand{\CT}{\mathcal{T}}
\newcommand{\rCT}{\mathring{\mathcal{T}}}

\newcommand{\dimY}{k}
\newcommand{\dimX}{d}
\def\bUpsilon{\boldsymbol{\Upsilon}}

\newcommand{\forest}[1]{\mathcal{F}_{\le #1}}
\newcommand{\rtree}[1]{\mathring{\mathcal{T}}_{\le #1}}
\newcommand{\tree}[1]{\mathcal{T}_{\le #1}}
\newcommand{\jbrac}[1]{ \langle  #1  \rangle}

\DeclareMathOperator*{\uint}{\scalerel*{\rotatebox{8}{$\!\scriptstyle\int\!$}}{\int}}
\usepackage{scalerel}
\usepackage{graphicx}
\usepackage{verbatim}

\def\all{\mathrm{\tiny all}}
\def\num{\#}

\newcommand{\graft}{\curvearrowright}

\usepackage{color}
\definecolor{darkred}{rgb}{0.9,0.1,0.1}
\definecolor{darkergreen}{rgb}{0.0, 0.5, 0.0}

\definecolor{darkblue}{rgb}{0.0, 0.2, 0.6}

\definecolor{orange}{rgb}{1, 0.5, 0.1}

\newcommand{\id}{\mathrm{id}}
\newcommand{\Char}{\mathrm{Char}}

\newcommand{\X}{\mathbf{X}}
\newcommand{\Y}{\mathbf{Y}}
\newcommand{\Z}{\mathbf{Z}}
\newcommand{\U}{\mathbf{U}}

\newcommand{\R}{\mathbb{R}}
\newcommand{\eqdef}{:=}

\def\TT{\widehat{\CT}}

\newcommand{\one}{\mathbf{1}}

\title[Bounds for RDEs with non-linear damping]{A priori bounds for rough differential equations with a non-linear damping term}

\author{Timothee Bonnefoi}
\address{Timothee Bonnefoi, \'Ecole Normale Sup\'erieure de Lyon
}
\email{timothee.bonnefoi@ens-lyon.fr}

\author{Ajay Chandra}
\address{Ajay Chandra, Imperial College London
}
\email{a.chandra@imperial.ac.uk}

\author{Augustin Moinat}
\address{Augustin Moinat, RADRiGS Slacklines
}
\email{augustin.moinat@gmail.com}

\author{Hendrik Weber}
\address{Hendrik Weber, University of Bath
}
\email{h.weber@bath.ac.uk}

\begin{document}
\begin{abstract}
We consider a rough differential equation with a non-linear damping drift term:
\begin{align*}
dY(t) = - |Y|^{m-1} Y(t) dt + \sigma(Y(t)) dX(t),
\end{align*}
where $m>1$,  $X$ is a (branched) rough path of arbitrary regularity $\alpha >0$,  and where $\sigma$ is smooth 
and satisfies an $m$ and $\alpha$-dependent growth property.
We show a strong a priori bound for $Y$, which includes the "coming down from infinity" property, i.e. the bound on 
$Y(t)$ for a fixed $t>0$ holds uniformly over all choices of initial datum $Y(0)$.

The method of proof builds on recent work on a priori bounds for the $\phi^4$ SPDE in arbitrary subcritical 
dimension \cite{ch2019priori}. A key new ingredient is an extension of the algebraic framework which 
permits to derive an estimate on higher order conditions of a coherent controlled rough path in terms of the 
regularity condition at lowest level.
\end{abstract}

\maketitle

\section{Introduction}

Rough path theory  \cite{L94,L98,G03,G06} was developed from the late 90s as a solution theory for  differential equations 
of the form
\begin{align}\label{RDE}
dY(t) = \sigma(Y(t)) dX(t) 
\end{align}
for irregular drivers $X(t)$, typically of H\"older regularity $\alpha<\frac12$. 
The principal example to have in mind is the case of 
random drivers. In particular, when $X$ is a Brownian motion, rough path theory  provides an alternative approach 
to the classical theory of  It\^o  stochastic differential equations (SDEs).
Since its inception in the late 90s the theory has been highly successful; early applications included the treatment of SDEs for correlated signals, such as 
fractional Brownian motion,  streamlined proofs for several classical results including large deviations, support theorems and the construction of stochastic flows for SDEs, 
see e.g. the monographs \cite{MR2036784,FrizVictoirBook,friz2014course}.
 More recent results include   the proof that certain deterministic fast-slow dynamical 
systems are  well-approximated by systems of SDEs
  \cite{kelly2016smooth} 
  or the development of a Malliavin calculus for stochastic Navier Stokes equations with 
multiplicative noise
\cite{gerasimovivcs2019hormander}. 

In this article we are concerned with deriving a priori bounds and the related issue of non-explosion for \eqref{RDE}. 
Rough path theory primarily addresses the small scale roughness of  solutions and the associated problem to define and 
control the rough integrals $\int \sigma(Y(t)) dX(t) $. There are relatively few results on the control of the global behaviour of solutions: 
in monographs, global existence for \eqref{RDE} is typically shown under the  restrictive assumptions that  $\sigma$ 
and all of its derivatives up to order $N $ are bounded  
 (see e.g.  \cite[Theorem 10.14]{FrizVictoirBook}, \cite[Theorem 8.4]{friz2014course}). Here and throughout the paper we denote by 
 $N$ the unique natural number satisfying
 \begin{equation*}
N\alpha \leq 1 < (N+1) \alpha.
\end{equation*}

Relaxing these growth conditions turns out to be intricate, and in fact  
global existence may fail, e.g. for bounded $\sigma$ with unbounded derivative 
\cite{Li20111407}. There are however some results which relax the boundedness assumptions on the full tensors $D^{\beta} \sigma$, $|\beta| \leq N$,
by showing that a control on  certain combinations of these derivatives is sufficient \cite{davie2008differential,Lejay2012,weidner2018geometric}. 

We are interested in rough differential equations with drift
\begin{equation*}
dY(t) = b(Y(t) )dt + \sigma(Y(t)) dX(t) \;.
\end{equation*}
The striking counter-example \cite[Section 3.3]{cox2013local} suggests that even in the case $\sigma = 1$ (additive noise) a
standard Lyapunov condition $b(y) \cdot y \leq C(|y|^2 +1) $ alone may not suffice to exclude finite-time explosion.
This issue was addressed in \cite{RS}, where it was shown that 
if $\sigma$ and its derivatives up to order $N$ are bounded, 
then the more restrictive Lyapunov condition
\begin{equation}\label{e:gen-Lyapunov}
b(y) \cdot y \leq C(|y|^2 +1) \qquad \text{and} \qquad \Big| b(y) - \frac{ (b(y) \cdot y ) y  }{|y|^2} \Big| \leq C (1+|y|) \;,
\end{equation}
 suffices to rule out finite-time explosion.
In this article we treat the case of super-linear inward drift $b$. 
 For simplicity, we restrict to $b$ of polynomial form, i.e. we deal with  
\begin{align}\notag
dY(t) &= - |Y(t)|^{m-1} Y(t)dt + \sigma(Y(t)) dX(t) \qquad t \in (0,1]\\
\label{RDE3}
Y(0) & = y_0 \in \R^k \; ,
\end{align} 
where $m>1$, $X$ is an $\R^d$ valued (branched) rough path of regularity $\alpha \in (0,1)$ in the sense of \cite{G06,HK,MR4013830} 
and $\sigma$ is assumed to be a $C^N$  map $\R^k \to \R^{k \times d}$, where $\R^{k \times d}$ denotes the space of $k \times d $ matrices.
This drift does satisfy \eqref{e:gen-Lyapunov} and therefore the results of \cite{RS} imply non-explosion if $\sigma$ and its derivatives are bounded. 
In this article, we use the damping effect of the drift to get stronger estimates on solutions  and to relax the boundedness conditions on $\sigma$. 
More specifically, we treat coefficients $\sigma$ satisfying either one of the following two growth assumptions:
\begin{assumption}[Bounded coefficients]
\label{Ass-Sigma-Bounded}
There exists a constant $C_{\sigma} < \infty$ such that for each multi-index $\beta  = (\beta_1, \ldots, \beta_k) \in \mathbb{N}^{k}$ 
of length $|\beta| := \sum_{j=1}^k \beta_j  \leq N$  we have
\begin{equation*}
\sup_{x \in\R^d} | \partial^\beta  \sigma(x) | \leq C_{\sigma}.
\end{equation*}
\end{assumption} 

\begin{assumption}[Coefficients of polynomial growth]
\label{Ass-Sigma-Polynomial}
There exists a $\gamma >0$ and a constant $C_{\sigma} < \infty$ such that for each multi-index $\beta   \in \mathbb{N}^{k}$ of length $|\beta|   \leq N$ and for all $x \in \R^k$
we have
\begin{equation*}
 | \partial^\beta\sigma (x)  |   \leq C_{\sigma} \jbrac{x}^{\gamma - |\beta|}\;,
\end{equation*} 
where above and below, $\jbrac{x} = (1+|x|^{2})^{\frac{1}{2}}$. 
\end{assumption}

\begin{remark}
Assumption~\ref{Ass-Sigma-Polynomial} is satisfied, for example, in the scalar case $d=k = 1$ and for  
\begin{equation*} 
\sigma(x) = \jbrac{x}^{\gamma}\;. 
\end{equation*}
\end{remark}
Our first main result is the following estimate:
\begin{theorem}[Coming down from infinity] \label{thm:main}
Let $\X$ be an $\alpha$-rough path  for some $\alpha \in (0,1)$ and let 
 $Y$ solve \eqref{RDE3}. 
\begin{enumerate}
\item If $\sigma$ satisfies the growth Assumption~\ref{Ass-Sigma-Bounded} we have for $t \in (0,1]$  
\begin{align}\label{e:thm-bound}
|Y(t) | \leq C \max \Big\{   t^{-\frac{1}{m-1}}   \; , \;\max_{h \in \rtree{N}}[\X:h]^{\frac{1}{ m\alpha|h|}}  \Big\} \;. 
\end{align}
\item  If $\sigma$ satisfies the growth Assumption~\ref{Ass-Sigma-Polynomial} for
\begin{align}\label{gamma-condition-m}
1 \leq \gamma < (m-1) \alpha +1 \;,
\end{align}
then we   have for $t \in (0,1]$
\begin{align}\label{e:thm-bound-bis}
|Y(t) | \leq C \max \Big\{  t^{-\frac{1}{m-1}}  \;, \; \max_{h \in \rtree{N}} [\X:h]^{\frac{1}{ ((m-1) \alpha- \gamma+1)|h|}} \Big\} \;.
\end{align}
\end{enumerate}
In both statements, $C< \infty$ denotes a constant that depends on $d,k,m,\alpha$ and $C_{\sigma}$.
\end{theorem}

The definition of a (branched) rough path is recalled in Definition~\ref{def-RP}, while the precise notion of solution to a rough differential equation is
 given in Definition~\ref{def:solution}. We mention right away that the solution $\Y$ of a rough differential equation as in Definition~\ref{def:solution} has to be viewed 
 as taking values in a larger space than $\R^k$ (a space spanned by decorated forests). Here $Y(t) \in \R^k$ is the "observable solution" which corresponds to the coefficient of 
 $\one$.  We write $|h|$ for the order of a tree, that is the number of its vertices, and  $\rtree{N}$ is the set of \emph{trees} of order $\leq N$, see Section~\ref{s:sets_of_trees}. 
The seminorm $[\X:h]$ 
controls the order bound for $\X$ at level $h$ - it can be thought of as a H\"older type bound on a (postulated) $|h|$-fold iterated integral of $X$ - see \eqref{order-norm}.
\begin{theorem}[Small times] \label{thm:main2}
Let $Y$ solve \eqref{RDE3}, for an $\alpha$-rough path $\X$. 
\begin{enumerate}
\item Let $\sigma$ satisfy the growth Assumption~\ref{Ass-Sigma-Bounded} and for $\epsilon_1 = \epsilon_1(m)$ and $\epsilon_2 = \epsilon_2 (\alpha,d,k, C_\sigma)$  small enough set 
\begin{equation}\label{e:def-T1T2-1}
T_1 = \epsilon_1 \frac{1}{\jbrac{y_0}^{m-1}}\;, \qquad \qquad   T_2 = \epsilon_2 \min   \bigg\{  \frac{1}{[\X \colon h  ]^{\frac{1}{|h| \alpha }}  }   \colon h \in  \rtree{N}   \bigg \}    \;.
\end{equation}
Then for $t \leq  \min \{ T_1,  T_2 \}$ we have
\begin{align}\label{e:small-time-bound1}
\jbrac{Y(t)} \leq 2  \jbrac{y_0}\;.
\end{align}
\item  Let $\sigma$ satisfy the growth Assumption~\ref{Ass-Sigma-Polynomial} for an exponent $\gamma$ satisfying \eqref{gamma-condition-m}, let $T_1$ be as in \eqref{e:def-T1T2-1} and  for 
$\epsilon_2 = \epsilon_2 (\alpha,d,k, C_\sigma)$  small enough set
\begin{equation}\label{e:def-T1T2-bis}
 T_2 = \epsilon_2  \Big( \frac{1}{\jbrac{y_0}}  \Big)^{\frac{\gamma-1}{\alpha}} \min   \bigg\{  \frac{1}{[\X \colon h  ]^{\frac{1}{|h| \alpha }}  }   \colon h \in  \rtree{N}   \bigg \}    \;.
\end{equation}
Then for $t \leq  \min \{ T_1,  T_2 \}$  we   have
\begin{align}\label{e:small-time-bound3}
\jbrac{Y(t) } \leq 2 \jbrac{y_0} \;.
\end{align}

\end{enumerate}
\end{theorem}
Theorems~\ref{thm:main} and~\ref{thm:main2} can be combined into a single bound.
\begin{corrolaire}\label{coro}
Let $Y$ solve \eqref{RDE3}, for an $\alpha$-rough path $\X$ for an $\alpha \in (0,1)$.
\begin{enumerate}
\item If $\sigma$ satisfies the growth Assumption~\ref{Ass-Sigma-Bounded} we have for $t \in [0,1]$  
\begin{align}\label{e:cor-bound}
|Y(t) | \leq C \max \Big\{  \min \{  t^{-\frac{1}{m-1}} , |y_0| \}   \; ,\max_{h \in \rtree{N}}[\X:h]^{\frac{1}{ (m-1)\alpha|h|}}  ,1  \Big\} \;. 
\end{align}
\item  If $\sigma$ satisfies the growth Assumption~\ref{Ass-Sigma-Polynomial} for $\gamma $ satisfying \eqref{gamma-condition-m},
then we   have for $t \in [0,1]$
\begin{align}\label{e:cor-bound-bis}
|Y(t) | \leq C \max \Big\{  \min \{ t^{-\frac{1}{m-1}} , |y_0|  \}   \;, \max_{h \in \rtree{N}} [\X:h]^{\frac{1}{ ((m-1) \alpha- \gamma+1)|h|}}  ,1 \Big\} \;.
\end{align}
\end{enumerate}
In both statements, $C< \infty$ denotes a constant that depends on $d,k,m,\alpha$ and $C_{\sigma}$.
\end{corrolaire}

\begin{remark}
Consider the case $\sigma = 0$, where \eqref{RDE3} reduces to the ordinary differential equation
\begin{align*}
\dot{Y}  &= - |Y|^{m-1} Y  \qquad \qquad t \in (0,1) \\
 Y(0)  & = y_0\;.
\end{align*}
The explicit solution of this equation is given by 
\begin{align}\label{e:solution_of_ODE}
Y(t) =  \frac{y_0}{|y_0|}   \Big( \frac{1}{(m-1) t  + |y_0|^{1-m} } \Big)^{\frac{1}{m-1}}\;,
\end{align}
and the elementary estimate  
\begin{align}
\frac{1}{C(m)} \min\{ |y_0|, t^{-\frac{1}{m-1}} \}  \leq   
  | \text{RHS of (\ref{e:solution_of_ODE}) }| 
\leq  C(m) \min\{ |y_0|, t^{-\frac{1}{m-1}} \} \, 
\end{align}
shows that the first term on the RHS of \eqref{e:cor-bound} and \eqref{e:cor-bound-bis} is optimal up to constants.
\end{remark}

\def\condsigma{\eqref{gamma-condition-m} }
\begin{remark}[On the optimality of the growth condition \condsigma   on $\sigma$] 
We compare our result to the case of the It\^o stochastic differential equation
\begin{equation*}
dY(t) = - |Y(t)|^{m-1} Y(t)dt + \sigma(Y(t)) dW(t) \qquad t \in (0,1] \; ,
\end{equation*}
for a Brownian motion $W$, for simplicity in the scalar case $k = n = 1$. 
It\^o's formula yields for any $p \geq 2$ 
\begin{align}\label{Ito}
d \frac{1}{p} |Y(t)|^p =  - |Y (t) |^{p-1 +m} dt +  dM_t + \frac{(p-1)}{2} |Y(t)|^{p-2} \sigma(Y(t))^2 dt \;,
\end{align}
for a (local) martingale $M_t$. Under Assumption~\ref{Ass-Sigma-Polynomial} the It\^o correction 
term on the RHS of \eqref{Ito} satisfies
\begin{align*}
\frac{(p-1)}{2} |Y(t)|^{p-2} \sigma(Y(t))^2 \leq \frac{(p-1)}{2} C_\sigma  |Y(t)|^{p-2 + 2 \gamma} \;.
\end{align*}
If the coefficient $\sigma$ satisfies the upper bound in  \eqref{gamma-condition-m} for an $\alpha < \frac12$, then the exponent on the RHS satisfies 
\[
p-2 + 2 \gamma < p -1 + m\; ,  
\]
which is precisely enough to control this term by the damping term $ - |Y (t) |^{p-1 +m}$. This calculation suggests 
the optimality of Assumption~\ref{Ass-Sigma-Polynomial}, at least in this case.
\end{remark}

\begin{remark}[More on the growth condition \condsigma   on $\sigma$] 
The reason for the limitation $\gamma \geq 1$ in \eqref{gamma-condition-m} is that in several points in the proof (see e.g. \eqref{an-example}, \eqref{e-Uestimate-simplified}) the estimate
\[
\sup_{s \in I} |Y(s)|^{(\gamma-1)|h|} \leq \| Y \|_{[t^\star, 1]}^{(\gamma-1)|h|}, 
\]
is used crucially. Here $I $ is some (short) interval contained in the (possibly much larger) $[t^\star, 1]$ and $|h| $ is a positive number. 
Clearly, this estimate only holds for $\gamma \geq 1$.
For $\gamma \leq 1$  our proof could be modified slightly to get the estimate 
\begin{align*}
|Y(t) | \leq C \max \Big\{  t^{-\frac{1}{m-1}}  \;, \max_{h \in \rtree{N}} [\X:h]^{\frac{1}{ ((m-1) \alpha)|h|}} \Big\} \;.
\end{align*}
Unfortunately, this bound does not reflect the behaviour of $\sigma$ and its derivatives at infinity in the exponents of $ [\X:h]$. 
\end{remark}

\begin{remark}[On optimality of exponents]
\label{exponents}
A particularly important example, where our result applies is the rough path $\X$ built on top of a \emph{fractional 
Brownian motion} $X$  of Hurst parameter $H > \frac14$. It is well-known, see e.g. \cite{MR2667703}, that in this case 
$\X$ is an $\alpha$-rough path for $\alpha < H$ and for $h \in \rtree{N}$ we have the Gaussian tail estimate
\begin{equation}\label{Fernique}
\E \Big[\exp\big( \epsilon [\X:h]^{\frac{2}{|h|}} \big)  \Big] < \infty,
\end{equation} 
for $\epsilon > 0$ small enough. In this case, the estimates \eqref{e:thm-bound} and \eqref{e:thm-bound-bis} imply 
stretched exponential tails for the solution $Y$ of the random rough differential equation \eqref{RDE3}. More precisely, 
\eqref{e:thm-bound}  combined with \eqref{Fernique} implies that under Assumption~\ref{Ass-Sigma-Bounded} we have
that for $\epsilon'>0$ small enough 
\begin{align*}
\E \Big[\exp\Big( \epsilon'  \sup_{t \in [\frac12,1] } |Y(t) |^{2\alpha m) }   \Big)  \Big] < \infty \;,
\end{align*}
while \eqref{e:thm-bound-bis} combined 
with \eqref{Fernique} implies that under Assumption~\ref{Ass-Sigma-Polynomial} and \eqref{gamma-condition-m} we have
\begin{align*}
\E \Big[\exp\Big( \epsilon'  \sup_{t \in [\frac12,1] } |Y(t) |^{2 (\alpha (m-1)  -\gamma +1)}   \Big)  \Big] < \infty \; .
\end{align*}
We do not know if these stretched exponential tails are optimal. The additive noise case, $\sigma =1$ was considered 
in \cite[Section 2]{MW}, and there the improved stochastic integrability 
\begin{align*}
\E \Big[\exp\Big( \epsilon'  \sup_{t \in [\frac12,1] } |Y(t) |^{2+2\alpha (m-1) }   \Big)  \Big] < \infty,
\end{align*}
was obtained. Additionally, in this case comparison with an It\^o SDE shows the optimality of exponent, at least in the 
case of Brownian motion. Martingale estimates 
suggest that in the case of a Brownian motion $W$ and in the case of bounded $\sigma$ the stochastic integrability 
in the multiplicative noise case should coincide with the additive noise case, but our argument does not imply this. 
The main technical reason  is, that in the proof we need to impose the condition \eqref{e:step2-condition} on the parameters
and this condition does not have a counterpart in the simpler argument for the additive noise case. 
\end{remark}

Our  argument follows the philosophy developed in  \cite{MW,moinat2018spacetime,ch2019priori} in the context of the dynamic $\phi^4$ model.  
These articles deal primarily with singular stochastic partial differential equations, but as already mentioned above in \cite[Section 2]{MW} the case of an ordinary differential equation with irregular driver is discussed. 
More precisely, an estimate for the equation  
\begin{equation}\label{NL1}
Y(t) = Y(0) - \int_0^t |Y(s)|^{m-1} Y(s) ds + Z(t)   %\qquad t \in [0,T],
\end{equation}
for $Z$ of H\"older regularity $\alpha \in (0,1)$ is derived. 
The argument  roughly goes as follows: first the equation is regularised by convolution with a mollification kernel at scale $L$, resulting in a 
non-closed equation for the regularised function $Y_L$,  
\[
\dot Y_L = -|Y_L|^{m-1}Y_L + Z_L + E_L
\]
  which involves an error term $E_L$ accounting for the fact that  non-linearity and smoothing do not commute. 
Afterwards, a standard ODE estimate  (see e.g. \cite[Lemma 3.8]{TW}) is applied to the regularised equation, and finally, the commutator error $E_L$  is controlled by a simple \emph{regularity estimate} for $Y$. 
The following lemma is an intermediate step in this analysis and the starting point for our argument:
\def\MWL4{\cite[Lemma 2.1, Intermediate result in Step 4]{MW}}
\begin{lemme}[\MWL4]
\label{lemma:MW}
Let $m>1$ and let $Y \in C([0,1], \R^d)$ satisfy equation \eqref{NL1}
for some $Z \in C([0,1], \R^d)$ with $Z(0)=0$. Then for $\alpha \in (0,1)$, $t \in (0,1]$ and $L \in (0,t)$  we have
\begin{align}
\notag
|Y(t)|  
\leq C \max \Big\{ 
	 & (t-L)^{-\frac{1}{m-1}}; 
	 \Big(  \sup_{s \in [L,t] } [Z]_{\alpha, [s-L,s]} L^{\alpha -1}\Big)^{\frac{1}{m}}; 
	\\&\label{NL2} 
	\Big(L^\alpha \| Y \|_{[0,t] }^{m-1}  \sup_{s \in [L,t] } [Y]_{\alpha,[s-L,s]}    \Big)^{\frac{1}{m}} ; 
	  \Big( L^\alpha  [Y]_{\alpha,[t-L,t]} \Big)    \Big\}.
\end{align}
Here $C =C(m)$ and
 \[
 \| Y \|_{[0,t] } = \sup_{s\in [0,t]} |Y(s)| \qquad \text{and} \qquad [Z]_{\alpha,[s-L,s]} = \sup_{s-L \leq s_1<s_2 \leq s } \frac{ |Z(t) - Z(s) |}{ |s_2- s_1|^{\alpha}}.
 \]
\end{lemme}
In order to apply this Lemma to the rough differential equation~\eqref{RDE3} it thus remains to control $\alpha$-H\"older norms of $Y$ and of the rough integral 
\begin{align*}
Z(t) = \int_0^t \sigma(Y(s)) dX(s) \; .
\end{align*}
Gubinelli's Sewing Lemma, see Proposition~\ref{prop:integral_bound} below, yields a control $Z$ in terms of the order bounds $[\X:f]$ as well as the
appropriate semi-norms $[\sigma(\Y):f]$ that control the regularity of the \emph{controlled path} $\sigma(\Y)$  at various levels $f$ ($f$ takes values in a set of \emph{decorated forests}, see below). 
A key step of our argument is the development of an algebraic framework which yields an explicit expression, Corollary~\ref{cor:niceformula}, for the 
remainder expressions (we note that a similar formula appeared earlier in \cite[Theorem 5.2]{G06}). 

This expression can then be combined with a Taylor remainder formula yielding a control 
on $[\sigma(\Y):f]$ in terms of $[Y:\one]$, the supremum norm of $Y$ as well as the various   $[\X:h]$. 
This bound furthermore depends linearly on $[Y:\one]$. 
The argument can then be closed, because the output of the Sewing Lemma gives a control on $[Y:\one]$ permitting to absorb this term, provided 
one works on an interval of small enough size. 
Summarising, this argument can then be used to control the terms $[Y]_\alpha$ and $[Z]_\alpha$ on the right hand side of \eqref{NL2} resulting in an estimate that can be iterated.

We finally point out some connections with past and recent work.
We first point out that in the setting of stochastic differential equations estimates   that show "coming down from infinity" akin to our Theorem~\ref{thm:main}  are known. See, for example, \cite[Proposition~1.2.7]{cerrai01}  for a such a statement, which seems however to impose a  slightly more restrictive growth condition on $\sigma$ than our \eqref{gamma-condition-m}.
Turning to the RDEs literature, it was already noticed in \cite[Proposition~3.8 and Lemma~3.10]{HK} that for controlled rough paths which solve RDEs one can control regularity at higher levels via control at the lowest level, but our use of grafting operations allows us to do this more explicitly and in a way that we expect will generalise from branched rough paths to the theory of regularity structures.  
As this work was nearing completion we also learned of the preprint \cite{BayerFrizTapia} which develops combinatorial/algebraic methods for comparing remainders quite close to what we use here based on brace algebras \cite{MR2724224}, \cite{OG}. 
\subsection{Outline of paper}
Section~\ref{s:BranchedRoughPaths} starts by reviewing the key elements of the theory of branched rough paths \cite{G06,HK} before closing with a discussion of coherence/elementary differentials - these allow us to describe what the expansions that solve fixed point problems need to look like.   
Section~\ref{s:coherence_and_remainder} starts by introducing grafting operations. 
We then show that this grafting operation behaves well with co-products (Theorem~\ref{thm:graft_adjoint}) and with coherence (Corollary~\ref{cor:niceformula}) - we then finally use these relations to obtain the remainder estimate Corollary~\ref{cor:one_controls_others}. 
Finally, Section~\ref{equ} finishes the proof of our main result, Theorem~\ref{thm:main}, and also proves Theorem~\ref{thm:main2} and Corollary~\ref{coro}. 
\subsection{Notation}
Given a natural number $m \ge 1$, we write $[m] = \{1,\dots, m\}$.
For any set $A$, we write $\Span(A)$ for the free vector space over $\mathbb{R}$ generated by $A$, that is all finite linear combinations of elements of $A$ with real coefficients. 
Given $p \in \mathbb{N}$ and a $p$-times differentiable functionable $\sigma: \mathbb{R}^{\dimY} \rightarrow \mathbb{R}^{\dimY}$, we write $D^{p}\sigma$ for the $p$-th derivative which is a $p$-linear form on $\mathbb{R}^{\dimY}$ taking values in $\mathbb{R}^{\dimY}$.
We also use standard notation for multi-indices. Given some finite set $A$, a multi-index $\ell = (\ell_{a} : a \in A) \in \mathbb{N}^{A}$, and sufficiently regular function $f: \mathbb{R}^{A} \rightarrow \mathbb{R}^{m}$, we write 
\[
\partial^{\ell}f
=
\big( \prod_{a \in A}\partial_{a}^{\ell_{a}} \big) f\;.\]
where $\partial_{a}$ is the partial derivative in the $a$-component.  
We also write $\ell! = \prod_{a \in A} (\ell_{a}!)$. 

For any $n \ge 1$, $\alpha \in (0,1)$,  interval $I \subseteq [0,1]$, and function $X:I \rightarrow \R^{n}$ we write 
\[
\jbrac{X}_{I} = \sup_{s \in I} \jbrac{X(s)}\;, \qquad 
 \| X \|_{I } = \sup_{s\in I} |X(s)|\;, \qquad \text{and} \qquad [X]_{\alpha,I} = \sup_{ s_1<s_2 \in I } \frac{ |X(t) - X(s) |}{ |s_2- s_1|^{\alpha}}\;.
 \]

Throughout the paper we write $\lesssim$ for $\leq C$, where the implicit constant $C$  may depend on $d,k,m,\alpha$ and $C_{\sigma}$.
and similarly for any interval $I \subset [0,1]$ and function $Y: I \rightarrow \R^{n}$ we set $\jbrac{Y}_{I} = \sup_{s \in I} \jbrac{Y_{s}}$
\section{Branched Rough paths}
\label{s:BranchedRoughPaths}
\subsection{Trees and forests}\label{ss:trees}
We start by inductively defining a set of rooted, decorated trees $\mathring{\mathcal{T}}_{\all}$ along with an associated set of forests $\mathcal{F}_{\all}$. 

A forest is a possibly empty, but finite, collection of rooted trees where we allow for multiple instances of the same tree. 
It is customary to write forests as (commutative) products, that is 
\begin{equation}\label{eq:forest}
f = \prod_{i \in I} h_{i}
\end{equation}
for some index set $I$ and $h_{i} \in \rCT_{\all}$. 
The empty forest will play a special role for us and we represent it with the symbol $\ind \in \mathcal{F}_{\all}$. 
The set $\mathring{\mathcal{T}}_{\all}$ is  generated by taking a forest of trees $f \in \mathcal{F}_{\all}$ and joining them all to a new root with a decoration $\mu \in [\dimX]$, we write the corresponding tree as $[f]_{\mu}$. 
Note that we have a natural identification $\rCT_{\all} \subset \mathcal{F}_{\all}$. 
We also remark that a forest $f \in \mathcal{F}_{\all}$ is also completely specified by a multi-index $f = (f_{h} : h \in \mathring{\mathcal{T}}_{\all} ) \in \mathbb{N}^{\mathring{\mathcal{T}}_{\all}}$. 

We also inductively define a notion of "order" $|\cdot|$ for forests by setting, for $f \in \mathcal{F}_{\all} = \sum_{h \in f} |h|$ and for $h = [f]_{\mu} \in \mathring{\mathcal{T}}_{\all}$, $|h| = |f|+1$. 
Here and in what follows, when we write $\sum_{h \in f}$ or $\prod_{h \in f}$, we are taking the product or sum over all the trees appearing on the right side of \eqref{eq:forest} \textbf{with multiplicity}. 

The order of a forest is then simply the number of nodes in that forest. 
The trees with a single node are just of the form $ [\ind]_{j}$ for some $j \in [d]$, below we draw such a tree by writing $\bullet_{j}$.
We would then have $[\bullet_{j}]_{i} = \<leg_black>$. 
We can draw forests by putting trees next to each other, with the order being unimportant, for instance
\[
\bullet_{k} \<leg_black> = \<leg_black> \bullet_{k} \in \mathcal{F}_{\all}
\quad
\textnormal{and}
\quad 
\Big[ 
\bullet_{k} \<leg_black> 
\Big]_{l} 
=
\<bigleg_black>
=
\<bigleg2_black>
\in
\mathring{\mathcal{T}}_{\all}
\;.
\]
We write $\#f = \sum_{h \in f} 1$ for the number of trees in $f$. 
Given two forests $f, \bar{f}$ we write $f \cdot \bar{f}$ for the forest defined by concatenating the two products of trees.
Note that $\ind$ plays the role of the unit for the forest product and that both $|\cdot |$ and $\#$ are additive over forest products.
By extending the forest product to linear combinations of forests via linearity, $\Span(\CF_{\all})$ becomes a unital commutative algebra with unit $\ind$.
By applying the product component-wise the same is true of $\Span(\CF_{\all}) \otimes \Span(\CF_{\all})$.
We write $\scal{\bullet}{\bullet}$  for the inner product on $\Span(\CF_{\all})$ with $\scal{f}{f'} = 1\{f = f'\}$ for all $f,f' \in \CF_{\all}$. 
We also denote by $\scal{\bullet}{\bullet}$ the corresponding inner product on $\Span(\CF_{\all}) \otimes \Span(\CF_{\all})$. 
Finally, we write, for any $n \in \mathbb{N}$, the linear projection $P_{\le n}: \Span(\CF_{\all}) \rightarrow \Span(\CF_{\all})$ that simply annihilate any forest $f \in \CF_{\all}$ with $|f| > n$. 

\subsection{Coproducts}
\begin{definition} 
The coproduct $\Delta$ is the map 
\[
\Delta: \Span(\CF_{\all}) \rightarrow \Span(\CF_{\all}) \otimes \Span(\CF_{\all})\;,
\] 
defined inductively:
\begin{itemize}
\item $\Delta \ind = \ind \otimes \ind$. 
\item For $[f]_{\mu}\in \mathring{\mathcal{T}}_{\all}$ we set $\Delta [f]_{\mu}=  [f]_{\mu}\otimes \ind + (\id \otimes [ \bullet ]_{\mu}) \Delta f$. 
\item For $f \in \CF_{\all}$, $\Delta f = \prod_{h \in f} \Delta h\;.$
\end{itemize} 
\end{definition} 
Note that this is just the Connes-Kreimer coproduct \cite{CK}. 
We note that the co-product preserves the number of nodes in the sense that, for $\bar{f},\tilde{f},f \in \CF_{\all}$, we have $\scal{ \Delta \bar{f}}{\tilde{f} \otimes f} \not = 0 \Rightarrow |\bar{f}| = |\tilde{f}| + |f|$.  

\subsection{Iterated integrals}

The set of trees  $\mathring{\mathcal{T}}_{\all}$ will be used to index postulated values of iterated integrals of $X = (X^{i})_{i=1}^{\dimX}$, for instance if $X$ is a smooth path, then all these iterated integrals can be canonically and the natural way to formalise our ``canonical'' assignment of iterated integrals is to define for $\mu  \in [\dimX]$ and  $h = [f ]_{\mu}$, 
\[
\scal{\X_{s,t}}{h}=\int_s^t
\scal{ \X_{s,r}}{f}dX_r^{\mu}\;.
\]
In this case we have
\begin{equation}\label{eq:example_trees}
\begin{split}
\scal{ \X_{s,t}}{ \bullet_{j}}  
=&
\int_{s}^{t}
dX^{j}_{r}
=
X^{j}_{t} - X^{j}_{s}\;, \quad 
\scal{\X_{s,t}}{\<leg_black>}
=
\int_{s}^{t}  \int_{s}^{r} dX^{j}_{u}  dX^{j}_{r}\;,\\
\scal{ \X_{s,t}}{ \<cube_black>} 
=&
\int_{s}^{t}
\big( \int_{s}^{r}  dX^{j}_{u_1} \big)
\big( \int_{s}^{r}  dX^{l}_{u_2} \big)
\big( \int_{s}^{r}  dX^{n}_{u_3} \big)
dX^{i}_{r}\;.
\end{split}
\end{equation}

Expansions indexed by such trees/iterated integrals naturally appear when solving differential equations by writing them as integral equations and iterating, for instance, in the case $\dimY = \dimX = 1$ we can generate a formal expansion for the solution for \eqref{RDE} as
\begin{equation*}%\label{eq:example_expansion}
\begin{split}
&Y_{t}- Y_{s}\\
&= \int_{s}^{t} \sigma(Y_{r}) dX_{r} \\
&=
\sigma(Y_{s})  \int_{s}^{t} dX_{r}
+
 D\sigma(Y_{s}) 
 \int_{s}^{t}  (Y_{r} - Y_{s}) dX_{r}
 +
 \frac{D^{2}\sigma(Y_{s})}{2} 
 \int_{s}^{t}  (Y_{r} - Y_{s})^{2} dX_{r} + \cdots\\
&=
\sigma(Y_{s})  \int_{s}^{t} dX_{r}
+
 D\sigma(Y_{s}) 
 \int_{s}^{t}   \int_{s}^{r} \sigma(Y_{u}) dX_{u} dX_{r}
 +
 \frac{D^{2}\sigma(Y_{s})}{2} 
 \int_{s}^{t}  \big(  \int_{s}^{r} \sigma(Y_{u}) dX_{u} \big)^{2} dX_{r} + \cdots\\
&
=
\sigma(Y_{s})  \int_{s}^{t} dX_{r}
+
\big( D\sigma(Y_{s}) \big) \sigma(Y_{s}) 
\int_{s}^{t} \int_{s}^{r}  dX_{u}dX_{r}
+
 \frac{D^{2}\sigma(Y_{s})}{2} \big( \sigma(Y_{s})  \big)^{2}
 \int_{s}^{t} \big( \int_{s}^{r} dX_{u} \big)^{2} dX_{r} +  \cdots\\
&
=
\sigma(Y_{s}) \scal{ \X_{s,t}}{\bullet}
+
\big( D\sigma(Y_{s}) \big) \sigma(Y_{s})  \scal{ \X_{s,t}}{\<legnolabel_black>}
+
 \frac{D^{2}\sigma(Y_{s})}{2} \big( \sigma(Y_{s})  \big)^{2}
 \scal{ \X_{s,t}}{\<squarenolabel_black>}
+  \cdots\;,
\end{split}
\end{equation*}
where we suppress the labels here in the one dimensional case - a similar expansion can be generated in the multi-dimensional case. 
Forests keep track of products of iterated integrals, for instance if $f$ is given as above then $\scal{\X_{s,t}}{f} = \prod_{h \in f} \scal{\X_{s,t}}{h}$. 

Note that if $X$ is a smooth function, then the Leibniz rule imposes relations on iterated integrals which would allow one to restrict to trees that are given by chains putting one in the setting of \emph{geometric rough paths}. 
However, the framework of branched rough paths allows us to make sense of this differential equation even when $X$ is too irregular for the right hand sides of \eqref{eq:example_trees} to be well-defined, the cost is that quantities like those on the left hand sides of \eqref{eq:example_trees} need to be postulated as part of the well-posedness problem. In practice such data is constructed via stochastic techniques. 
In particular, the natural choice of values for these iterated integrals may not satisfy Leibniz rule, which is why we allow for trees that branch instead of just chains. 
\subsection{Important sets of trees and forests} 
\label{s:sets_of_trees}
We define, for $j \in \{N-2,N-1,N\}$, 
\begin{equation*} 
\begin{split}
\rtree{j} &= \{ h \in \mathring{\mathcal{T}}_{\all}: |h| \le j \}\;,\;
\tree{j} = \rtree{j} \cup \{\ind\}\;\\
\CF &=\{ f \in \mathcal{F}_{\all}: h \in f \Rightarrow h \in \rCT_{\le N} \}  \subset \mathcal{F}_{\all}\;,\;
\forest{j} = \{f \in \mathcal{F}: |f| \le j \}\;.
\end{split}
\end{equation*}
Note that $\Span(\CF)$ is a sub-algebra of $\Span(\CF_{\all})$.   
We write $\Span(\CF)^{\ast}$ for all linear functionals from $\Span(\CF)$ to $\R$. Given $u \in \Span(\CF)^{\ast}$ and $a \in \Span(\CF)$ we overload notation and write $\scal{u}{a}$ for the corresponding duality pairing.
We define $\Char \subset \Span(\CF)^{\ast}$ to be the set of real characters on $\Span(\CF)$, that is the set of all algebra homomorphisms from $\Span(\CF)$ to $\R$.
\subsection{Character paths}
\begin{definition}
A character path $\X$ is a function $\X_{\bullet,\bullet}: [0,T]^{2} \rightarrow \Char$, written $
(s,t) \mapsto
\X_{s,t}$. 
\end{definition}
\begin{remark}\label{rem:trees_versus_forests2}
We can now motivate some of the definitions of sets of trees and forests we introduced earlier. 

The set  $\tree{N-1}$ is used as index set for local expansions of the solution $Y$ to \eqref{RDE3}, while $\forest{ N-1} $ is used as index set for expansions of $\sigma_{\mu}(Y(t))$ for $\mu \in [\dimX]$.

The crucial data regarding ``iterated integrals'' of the $X$ is indexed by the trees $\rtree{N}$. 
In particular, while the trees of order $N$ never explicitly appear in expansions, the fact that $\X$ is defined on those trees and satisfies an appropriate H\"{o}lder bound there is crucial - see the Sewing Lemma - Proposition~\ref{prop:integral_bound} (\cite{G03,G06}).

Extending $\X$ multiplicatively means that $\CF$ is the natural set of objects for $\X$ to act on.   
We note that any character path $\X$ satisfies $\scal{\X_{s,t}}{\ind} = 1$. 
\end{remark}
\begin{definition}
A character path satisfies Chen's relation if, for every $s,u,t \in [0,T]$ and $h \in \CT$, 
\begin{equation*}%\label{eq:chen}
\scal{(\X_{s,u} \otimes \X_{u,t})}{\Delta h} =
\scal{\X_{s,t}}{h} \;.	
\end{equation*}
\end{definition}
Note that above, we overload the notation $\scal{\bullet}{\bullet}$ for tensor products by setting $\scal{ u \otimes v}{a \otimes b} = \scal{u}{a}\scal{v}{b}$ and then extending by linearity. 
Given a character path $\X$ and $f \in \CF$ we define 
\begin{equation}
\label{order-norm}
[\X: f] :=
\sup_{
0 \le s,t \le 1
\atop
s \not = t}
\frac{|\scal{\X_{s,t}}{f}|}{|s-t|^{\alpha |f|}}\;.
\end{equation}
The exponent $\alpha$ above corresponds to the imposed regularity condition on the driving noise. 
Note that one has the obvious bound $[\X: f] \le \prod_{h \in \CF}
[\X : h]$.
\begin{definition}
\label{def-RP}
A (branched) $\alpha$-rough path is a character path $\X$ that satisfies Chen's relation and has the property that $[\X:h] < \infty$ for every $h \in \rCT$.
\end{definition} 
We usually drop the word branched below, just calling $\X$ an $\alpha$-rough path. 
\begin{proposition}
If $\X$ is the canonical lift of a smooth function $X$ then $\X$ is an $\alpha$-rough path for any $\alpha \in (0,1]$.  
 \end{proposition}
\subsection{Controlled paths}
\begin{definition}
We call a map $ \Y:[0,T] \rightarrow \Span(\tree{ N-1})^{\dimY}$ a tree path and a map $\Y:[0,T] \rightarrow \Span(\forest{N-1})^{\dimY}$ a forest path. 
\end{definition}
We apply operations to tree/forest paths component wise, writing expressions like $\Delta \Y_{s} \in (\Span(\forest{N-1}) \otimes \Span(\forest{ N-1}))^{\dimY}$ and so on. 

We usually use the removal of boldface to indicate the ``$\ind$ - component'' of a tree path or forest path $\Y$, for instance we write $Y_{s} = \scal{\ind}{\Y_{s}} \in \R^{\dimY}$ and define $\delta \Y_{s}$ so that 
\begin{equation}\label{decomposition}
\Y_{s}  = Y_{s} \ind + \delta \Y_{s}\;.
\end{equation}
\begin{remark}
Following the thread from Remark~\ref{rem:trees_versus_forests2}, a tree path $\Y$ is used as an expansion for the solution $Y$ to \eqref{RDE3} while, for each $\mu \in [\dimX]$, a forest path  $\sigma_{\mu}(\Y)$ will be used as an expansion of $\sigma_{\mu}(Y)$.
\end{remark}
\begin{definition}\label{def-Ynorm}
Fix $I \subset [0,1]$ and let $ \X$ be an $\alpha$-rough path. 
Given a forest/tree path $ \Y$ we define, for any $f \in\forest{N-1}$, 
\[
[\Y:f]_{I}
=
\sup_{s,t \in I \atop 0 < |s-t| < 1}
\frac{|R^{\Y,f}_{s,t}|}{|t-s|^{(N-|f|)\alpha}}
\]
where the $f$-remainder of $\Y$ between $s$ and $t$, denoted $R^{\Y,f}_{s,t}$, is given by
\begin{equation*}%\label{e:def-remainder}
R^{\Y,f}_{s,t} \eqdef \scal{f}{ \Y_t} - \scal{ \X_{s,t} \otimes  f}{ \Delta \Y_s} \in \mathbb{R}^{\dimY}\;.
\end{equation*}
When $I = [0,1]$ we suppress it from that notation, i..e writing $[\Y:f]_{[0,1]} = [\Y:f]$. 
We call $\Y$ an $\X$-controlled rough forest/tree path if $\max \{[\Y :f]: f \in\forest{N-1}\} < \infty$. 
\end{definition}
As in \cite{HK} we will define the composition of a function $g \in C^{N-1}(\mathbb{R}^{\dimY},\mathbb{R}^{\dimY})$ with a tree path $\Y$ via appealing to the Taylor expansion of $g$ about $Y_{\bullet}$. 
In particular, we define the forest path $g(\Y)$ by setting
\begin{equation}\label{eq:non-linearity_of_tree_path}
g(\Y)_{s}
=
g( Y_{s})  \ind +
P_{ \le N-1}
\sum_{p = 0}^{N-1}
\frac{1}{p!}
D^{p}
g(Y_{s}) \big[
\big(\underbrace{\delta \Y_{s},\dots,\delta \Y_{s}}_{p\textnormal{-times}} \big)
\big]\;.
\end{equation}
\begin{remark}
Note that in \eqref{eq:non-linearity_of_tree_path} we naturally view $p$-linear forms on $\R^{\dimY}$ as $p$-linear forms on $\Span(\rtree{N-1})^{\dimY}$, for instance
\[
D^{p}g(Y_{s})
\big[
\big(\underbrace{\delta \Y_{s},\dots,\delta \Y_{s}}_{p\textnormal{-times}} \big)
\big]
=
\sum_{h_{1},\dots,h_{p} \in \rtree{N-1}}
D^{p}g(Y_{s}) \big[ \big( \scal{\Y_{s}}{h_{1}}, \dots, \scal{\Y_{s}}{h_{p}} \big) \big]
\prod_{i=1}^{p} h_{i}\;.
\]
\end{remark}

\subsection{Integration of controlled paths} 
We now explain how we define our rough integral that is needed to define our fixed point problem. Given a forest path $\U$, $\mu \in [\dimX]$, and $s,t \in [0,1]$, we define
\[
\Xi^{\U,\mu}_{s,t}
=
\sum_{f \in\forest{N-1}}
\scal{f}{\U_s}\scal{\X_{s,t}}{[f]_{\mu}}\;.
\]
The quantity $\Xi_{s,t}$ should be seen as a single term in a Riemann -type approximation for the rough integral.   
We can now state Gubinelli's sewing lemma \cite{G03,G06}.
\begin{proposition}
\label{prop:integral_bound}
Let $\X$ be a branched rough path and $\U$ be an $\X$-controlled forest path. Then, for any $0 \le s \le t \le 1$, we have   
\begin{equation}\label{eq:integral as a value}
\lim_{ |\mathcal{P}| \rightarrow 0} 
\sum_{\{s_{n},s_{n+1} \} \in \mathcal{P}}
\Xi^{\U,\mu}_{s_{n},s_{n+1}}
=:
\int_{s}^{t}
\U_{r} \mathrm{d}\X^{(\mu)}_{r}
\end{equation}
where the limit above is over partitions $\mathcal{P}$ of $[s,t]$ with mesh size going to $0$. 

We have the estimate
\begin{equation}\label{eq:roughintegral_bound}
\Bigg| \int_s^t \U_r \mathrm{d}\X^{(\mu)}_r - \Xi^{\U,\mu}_{s,t} \Bigg|
\le C(\alpha)|t-s|^{(N+1)\alpha}
\sum_{f\in\forest{N-1}} \sum_{\mu\in[\dimX]}
 [\X: [f]_{\mu}] [\U : f  ]_{[s,t]} \; .
\end{equation}
\end{proposition}

Note that $\int_s^t \U_r \mathrm{d}\X^{(\mu)}_r$ is  an element of $\R^{\dimY}$, not a tree/forest path. Our notion of a solution to \eqref{RDE3} will itself be an expansion so we also want to make sense of the integration of an $\X$-controlled forest path as a tree path rather than just an element  $\R^{\dimY}$.  Therefore we define, given a branched rough path $\X$, an $\X$-controlled tree path $\Y$,  and a tuple of functions $\sigma = (\sigma_{\mu}: \mu \in [\dimX])$ with $\sigma_{\mu} \in C^{N-1}(\R^{\dimY}, \R^{\dimY})$,  the new tree path
\begin{equation}\label{eq:integral as a path}
\Z_{\bullet}
=
\uint_{0}^{\bullet}
\sigma(\Y)_{r} \mathrm{d} \X_{r} 
=
\sum_{\mu \in [\dimX]}
\Big(
\int^{\bullet}_{0}
\sigma_{\mu}(\Y)_{r} \mathrm{d}\X^{(\mu)}_{r} 
\Big) \ind
+
\sum_{f \in\forest{N-2}}
\scal{f}{\sigma_{\mu}(\Y)_{\bullet}} 
[f]_{\mu}\;.
\end{equation}
Note that we  typographically distinguish the integral $\int$ appearing in \eqref{eq:integral as a value} from the integral $\displaystyle \uint$ appearing in \eqref{eq:integral as a path} since they produce different sorts of objects. 

We can now state a precisely what we mean by a ``solution'' to \eqref{RDE3}.

\begin{definition}\label{def:solution}
Given a branched rough path $\X$ we say that an $\X$-controlled tree path $\Y$ is a solution to \eqref{RDE3} if we have, for $0  \le s \le 1$, 
\begin{equation}\label{eq:solution}
\Y_{s}
=
y_0 \ind 
-
\Big(\int_{0}^{s} |Y_{r}|^{m -1} Y_{r} \mathrm{d}r\Big)
\ind
+
\Z_{s}\;.
\end{equation}
where $\Z_{s}$ is defined as in \eqref{eq:integral as a path}. 
\end{definition}

\subsection{Coherence}
A necessary (but not sufficient) condition for the tree path $\Y$ to be a solution in the sense of Definition~\ref{def:solution} is that:
\begin{equation}\label{eq:lhs_to_rhs}
\forall [f]_{\mu} \in \rtree{ N-1}, \ s \in (0,T],\ \textnormal{ one has }
\scal{f}{\sigma_{\mu}(\Y_{s})}
=
\scal{[f]_{\mu}}{\Y_{s}}\;.
\end{equation}
The condition \eqref{eq:lhs_to_rhs}  is algebraic and completely local in time.  
In order to be able to conclude that $\Y$ is a solution in the sense of Definition~\ref{def:solution} one would additionally have to impose that $\Y$ is $\mathbf{X}$-controlled and that, for all $s \in (0,T]$, the $\ind$ coefficient on both sides of \eqref{eq:solution} are the same. 

Given $Y_{s}$ as input, one can use \eqref{eq:lhs_to_rhs} to calculate either side of \eqref{eq:lhs_to_rhs} inductively in $|f|$.
This motivates the following definition.
\begin{definition}\label{def:coherence_coeff}
We define the collection of functions
\[
(\Upsilon_{\mu}[f](\cdot): f \in\forest{N-1} 
\;,\;
\mu \in [\dimX])\;,\;
\Upsilon_{\mu}[f](\cdot):\R^{\dimY} \rightarrow \R^{\dimY}\;,
\]
inductively, as follows. 
Given 
\[
f = \prod_{i \in I} [f_{i}]_{\mu_{i}}
\] 
we set 
\begin{equation}\label{eq:upsilon_def}
\Upsilon_{\mu}[f](\cdot)
=
D^{|I|}
\sigma_{\mu} (\cdot) \big[ \big( \Upsilon_{\mu_{i}}[f_{i}](\cdot)\big)_{ i \in I } \big] \;.
\end{equation} 
The definition given above is inductive in $|f|$ and also fixes the base case $f = \ind$ where we have $\Upsilon_{\mu}[\ind](\cdot) = \sigma_{\mu}(\cdot)$. 
\end{definition}
\begin{remark}
The functions $\Upsilon_{\mu}[f]$ are also known as \emph{elementary differentials} in the analysis of Runge-Kutta methods, see \cite{Brouder}. 
\end{remark}

These elementary differentials usually appear with symmetry factors. 
We define $S: \CF_{\all} \rightarrow \mathbb{N}$ via recursively setting, $S(\ind) = 1$, and 
 \[
S(f) = \prod_{h \in \rCT_{\all}} (f_{h}!) S(h)^{f_{h}}\;,
\]
where we recall that $f_{h} \in \mathbb{N}$ is the number of instances of $h$ in $f$. 

A straightforward inductive proof then gives the following lemma. 

\begin{lemme} Given a tree path $\Y$, the condition \eqref{eq:lhs_to_rhs} is equivalent to 
\begin{equation}\label{eq:coherence_prop}
\forall [f]_{\mu} \in \mathring{\CT}_{\le N-1}, \ s \in (0,T],\ \textnormal{ one has }
\scal{f}{\sigma_{\mu}(\Y_{s})}
=
\scal{[f]_{\mu}}{\Y_{s}} = \frac{1}{S(f)}\Upsilon_{\mu}[f](Y_{s})\;.
\end{equation}
\end{lemme}
\begin{definition}
We call a tree path $\Y$ that satisfies \eqref{eq:lhs_to_rhs} (or equivalently \eqref{eq:coherence_prop}) \emph{coherent}. 
\end{definition}

\begin{remark}
We extend $\Upsilon_{\mu}[\bullet]$ from $\CF_{\all}$ to $\Span(\CF_{\all})$ by linearity. 
We also define $\Upsilon[ \bullet]$ on $\Span(\mathring{\CT}_{\all})$ by setting, for $h = [f]_{\mu} \in \mathring{\CT}_{\all}$, $\Upsilon[h] = \Upsilon_{\mu}[f]$ and then again extend linearly. 
We can then write \eqref{eq:upsilon_def} as
\begin{equation}\label{eq:upsilon_def2}
\Upsilon \big[[f]_{\mu} \big](\cdot) = \Upsilon_{\mu}[f](\cdot)
=
D^{\num f}
\sigma_{\mu} (\cdot) \big[ \big( \Upsilon[h](\cdot)\big)_{h  \in f } \big]\;.
\end{equation}
This duplication of notation (giving equivalent definitions of $\Upsilon_{\mu}$ on forests and $\Upsilon$ on trees) is 
convenient because trees appear for the expansion of $Y$ while forests appear for the expansion of the right hand side of the equation. 
\end{remark}

As an example, we compute 
\[
\Upsilon_{l}\Big[ 
\bullet_{k} \<leg_black> 
 \Big](x)
  = 
\Upsilon\Big[ \<bigleg_black> \Big](x)
=
D^2 \sigma_{l}(x)\big[ \big( \sigma_{k}(x)  , D \sigma_{i}(x)[ \sigma_{k}(x) ] \big) \big]\;.
\]

We close this section with the following remark. 
\begin{remark}\label{rem:drift}
Note that the drift term $- |Y(t)|^{m-1} Y(t)dt$ didn't play much of a role in this section, in particular it has no direct effect on the coherence coefficients $\Upsilon_{\mu}[\cdot](\cdot)$. 
However, since$ \big| \int_{s}^{t} |Y(r)|^{m-1} Y(r)dr \big| \lesssim |s-t|$, this term is negligible compared to the other terms in our local expansion for $Y$. 
In particular, by virtue of \eqref{eq:solution}, we have  % 
\begin{equation}\label{eq:remainder_relation}
R^{\Z,\ind}_{s,t} = 
 Z_t - Z_s
 - \scal{\X_{s,t}}{\delta\Z_{s}}
 = R_{s,t}^{\Y,\ind} + \int_s^t |Y(r)|^{m-1} Y(r) dr \;,
 \end{equation}
 where  $\delta \Z$ is defined via the same convention as in \eqref{decomposition} and we used that, by coherence,  
 \begin{equation*}
\scal{\X_{s,t}}{\delta\Z_{s}}= 
\sum_{h \in \rtree{N-1}} 
 \scal{\X_{s,t}}{h}\scal{h}{\Z_{s}} 
 \;
 =
 \sum_{h \in \rtree{N-1}} 
 \scal{\X_{s,t}}{h}\scal{h}{\Y_{s}}
\;.
\end{equation*} 
The role of the drift becomes important only in the final step of our argument in Section~\ref{sec:final}.
\end{remark}

\section{Grafting, coherence, and remainder estimates}
\label{s:coherence_and_remainder}
\subsection{Grafting operations}\label{subsec:grafting_op}
We introduce grafting operations which will later be shown to be ``adjoints'' of the co-product under a particular inner product we introduce later. 

In words, given $\tilde{f}, f \in \CF_{\all}$, the grafting of $\tilde{f}$ onto $f$ sums over all the ways of attaching each of the trees in $\tilde{h} \in \tilde{f}$ to a vertex of $f$ with an edge, or concatenating $\tilde{h}$ to the forest $f$. 
More precisely, given 
\begin{equation}\label{eq: productform}
\tilde{f} = \prod_{j \in J} \tilde{h}_{j} \in \CF_{\all} 
\textnormal{ and }
f = \prod_{i \in I} [f_{i}]_{\mu_i} \;.
\end{equation}
we inductively define
\begin{equation}\label{eq:inductive_def_grafting}
\tilde{f} \graft f
=
\sum_{K \subset J}
\Big(
\prod_{j \in J \setminus K}
\tilde{h}_{j}
\Big)
\sum_{ \theta: K \rightarrow I}
\Big(
\prod_{i \in I}
\Big[
\Big(
\prod_{k \in \theta^{-1}(i)}
\tilde{h}_{k}
\Big)
\graft f_{i}
\Big]_{\mu_i}
\Big) \in \CF_{\all}\;.
\end{equation}
The definition above is inductive in $|f|$, with the base case corresponding to $f = \ind$ for which we have $I = \emptyset$ in \eqref{eq: productform} and so one must choose $K = \emptyset$ in the first sum of \eqref{eq:inductive_def_grafting} -- in particular we have $\tilde{f} \graft \ind = \tilde{f}$. 
Additionally, we note that $\scal{\tilde{f} \graft f}{\bar{f}} \not = 0 \Rightarrow |\tilde{f}| + |f| = |\bar{f}|$. 
We give an example grafting computation below:
\begin{equation*}
\begin{split}
\<square_black> \bullet_{k} 
\graft \bullet_{l} \<leg_black>
&=
\<square_black> \bullet_{k}  
 \bullet_{l} \<leg_black>
 +
\<square_black> \ 
\big(
\<leg2_black> \<leg_black> 
+
 \bullet_{l} 
 (
 \<square2_black>
 +
 \<longleg_black>)
\big) \\
{}& \quad +
 \bullet_{k}  
 \big(
\<tallsquare_black>
\<leg_black>
+
 \bullet_{l} 
(
\<squareleg1_black>
+
\<squareleg2_black>
)
\big)\\
{}& \quad +
\<tallsquare_black>
\big(
 \<square2_black>
 +
 \<longleg_black>
\big)
+
\<leg2_black>
\big(
\<squareleg1_black>
+
\<squareleg2_black>
\big)\;.
\end{split}
\end{equation*}

\begin{remark} The operation $\tilde{f} \graft f$ is a decorated version of the ``grafting forests over forests'' operation of \cite[Section~1.5]{foissy}, this is essentially a version of the Grossman-Larson product \cite{GL-Hopf}.   
\end{remark}

\subsection{Inner products and the adjoint relation}\label{subsec:inner_product}
\begin{definition}\label{def:inner_product_two}
We define a new inner product on $\Span(\CF_{\all})$ by inductively setting
\begin{itemize}
\item $ \bscal{\ind}{ \ind} = 1$
\item $\bscal{ [f]_{\mu}} {[g]_{\mu'}} = 1\{ \mu = \mu' \} \bscal{f}{g}$
\item For $h_{i}, g_{j} \in \mathring{\CT}_{\all}$, 
\[
\bscal{ \prod_{i \in I} h_{i}}{\prod_{j \in J} g_{j}} = \sum_{\theta: I \rightarrow J \atop \mathrm{biject.}}
\prod_{i \in I} \bscal{h_{i}}{g_{\theta(i)}}\;,
\]
\end{itemize}
We extend this inner product to $\Span(\CF_{\all}) \otimes \Span(\CF_{\all})$ by setting, for $f,g,\bar{f},\bar{g} \in \CF_{\all}$, 
\[ 
\bscal{ f \otimes g}{ \bar{f} \otimes \bar{g}} 
= 
\bscal{f}{ \bar{f}}
\bscal{g}{ \bar{g}}\;.
\]
\end{definition}
\begin{remark}
Note that, for $f,g \in \CF_{\all}$, 
\begin{equation}\label{eq:inner_product_relation}
\bscal{f}{g} 
=
1\{f = g\} S(f) = S(f) \scal{f}{g}\;.
\end{equation}
In particular, \eqref{eq:coherence_prop} can be rewritten as :
\begin{equation*}
\forall [f]_{\mu} \in \mathring{\CT}_{\le N-1}, \ s \in (0,T],\ \textnormal{ one has }
\bscal{f}{\sigma_{\mu}(\Y_{s})}
=
\bscal{[f]_{\mu}}{\Y_{s}} = \Upsilon_{\mu}[f](Y_{s})\;.
\end{equation*}
\end{remark}

The following lemma is straightforward to prove from our definition of $\bscal{\bullet}{\bullet}$. 
\begin{lemme}\label{lemma_tensor_identity}
Given any $\tilde{h}_{j},  h_{i} \in \mathring{\CT}$ and also $a_{n} \in \Span(\CF_{\all}) \otimes \Span(\CF_{\all})$ one has
\[
\Bscal {\big( \prod_{j \in J} \tilde{h}_{j}  \big) \otimes  \big( \prod_{i \in I} h_{i} \big)}{ 
 \prod_{n \in N} a_{n}}
=
\sum_{
\substack{\theta_J: J \rightarrow N \\
\theta_I: I \rightarrow N}}
\prod_{n \in N}
\Bscal{ 
\big( \prod_{ j \in \theta_{J}^{-1}(n)} \tilde{h}_{j} \big)
\otimes
\big( \prod_{ i \in \theta_{I}^{-1}(n)} h_{i} \big)
}{
a_{n}
}
\]
\end{lemme}
Our main purpose for introducing the inner products $\bscal{\bullet}{\bullet}$ is that it is convenient for stating and proving following theorem. 
\begin{theorem}\label{thm:graft_adjoint}
For any $\tilde{f},f, g \in \CF_{\all}$, we have
\begin{equation}\label{eq:graft_adjoint}
\bscal{ \tilde{f} \graft f}{ g } = \bscal{\tilde{f} \otimes f}{ \Delta g}\;.
\end{equation}
\end{theorem}
\begin{proof}
We give an inductive proof with the induction being in $|f|$. For the base case we have $f = \ind$ in which case \eqref{eq:graft_adjoint} is immediate.  
We now prove the inductive step. 
We write $g = \prod_{n \in N} [g_{n}]_{\bar{\mu}_{n}}$ and also write $\tilde{f}$ and $f$ as in \eqref{eq: productform}. We then have
\begin{equation*}
\begin{split}
\bscal{ \tilde{f} \graft f}{ g}
=& 
\sum_{K \subset J}
\sum_{ \theta: K \rightarrow I}
\Bscal{
\Big(
\prod_{j \in J \setminus K}
\tilde{h}_{j}
\Big)
\prod_{i \in I}
\Big[
\Big(
\prod_{k \in \theta^{-1}(i)}
\tilde{h}_{k}
\Big)
\graft f_{i}
\Big]_{\mu_i}}
{
\prod_{n \in N} [g_{n}]_{\bar{\mu}_{n}}
}\\
=&
\sum_{K \subset J}
\sum_{\sigma: (J \setminus K) \sqcup I \rightarrow N
\atop
\textnormal{biject.}}
\sum_{ \theta: K \rightarrow I}
\Big(
\prod_{j \in J \setminus K}
\Bscal{ \tilde{h}_{j}}
{ [g_{\sigma(j)}]_{\bar{\mu}_{\sigma(j)}}}
\Big)\\
{}&\quad
\times
\Big(
\prod_{i \in I}
\Bscal{
\Big[ 
\Big(
\prod_{k \in \theta^{-1}(i)}
\tilde{h}_{k}
\Big)
\graft f_{i}
\Big]_{\mu_i}}
{ 
[g_{\sigma(i)}]_{\bar{\mu}_{\sigma(i)}}
}
\Big)\\
=&
\sum_{K \subset J}
\sum_{\sigma: (J \setminus K) \sqcup I \rightarrow N
\atop
\textnormal{biject.}}
\sum_{ \theta: K \rightarrow I}
\Big(
\prod_{j \in J \setminus K}
\Bscal{ \tilde{h}_{j}}
{ [g_{\sigma(j)}]_{\bar{\mu}_{\sigma(j)}}}
\Big)\\
{}&
\quad \times
\Big(
\prod_{i \in I}
1\{
\mu_i = \bar{\mu}_{\sigma(j)}\}
\Bscal{
\Big(
\prod_{k \in \theta^{-1}(i)}
\tilde{h}_{k}
\Big)
\otimes  f_{i}
}
{\Delta g_{\sigma(i)}
}
\Big)\;.
\end{split}
\end{equation*}
Above, in the second equality we used the third identity of Definition~\ref{def:inner_product_two}, and in the third equality we used our the third identity of Definition~\ref{def:inner_product_two} followed by our induction hypothesis. 

One the other hand, $\Delta g
=
\sum_{M \subset N}
\big ( \prod_{m \in M} [g_{m}]_{\bar{\mu}_{m}} \otimes \ind \Big) 
\Big( \prod_{n \in N \setminus M}  (\id \otimes [ \bullet ]_{\bar{\mu}_{n}}) \Delta g_{n} \Big)$ 
which gives
\begin{equation*}
\begin{split}
{}&\bscal{\tilde{f} \otimes f}{\Delta g}\\
=& 
\sum_{M \subset N}
\Bscal{
\tilde{f} \otimes f}{
\Big( \prod_{m \in M} [g_{m}]_{\bar{\mu}_{m}} \otimes \ind \Big) 
\Big( \prod_{n \in N \setminus M}  (\id \otimes [ \cdot]_{\bar{\mu}_{n}}) \Delta g_{n} \Big)
}\\
=& 
\sum_{M \subset N}
\sum_{\theta_{J}: J \rightarrow N}
\sum_{\theta_{I}: I \rightarrow N }
\Big(
\prod_{m \in M}
\Bscal{
\Big(
\prod_{j \in \theta_{J}^{-1}(m)}
\tilde{h}_{j}
\Big)
\otimes
\Big(
\prod_{i \in \theta_{I}^{-1}(m)}
[f_{i}]_{\mu_i}
\Big)}
{
[g_{m}]_{\bar{\mu}_{m}} \otimes \ind
}
\Big)\\
{}&\quad \times
\Big(
\prod_{n \in N \setminus M}
\Bscal{
\Big(
\prod_{j \in \theta_{J}^{-1}(n)}
\tilde{h}_{j}
\Big)
\otimes
\Big(
\prod_{i \in \theta_{I}^{-1}(n)}
[f_{i}]_{\mu_i}
\Big)}
{
  (\id \otimes [ \cdot]_{\bar{\mu}_{n}}) \Delta g_{n}}
\Big)\\
=& 
\sum_{M \subset N}
\sum_{\theta_{J}: J \rightarrow N}
\sum_{
\substack{
\theta_{I}: I \rightarrow N \setminus M \\  \textnormal{biject.}}}
\Big(
\prod_{m \in M}
\Bscal{
\Big(
\prod_{j \in \theta_{J}^{-1}(m)}
\tilde{h}_{j}
\Big)}
{
[g_{m}]_{\bar{\mu}_{m}} 
}
\Big)\\
{}&\quad \times
\Big(
\prod_{i  \in I}
1 \{ \mu_{i} = \bar{\mu}_{\theta_{I}(i)}\}
\Bscal{
\Big(
\prod_{j \in \theta_{J}^{-1}(\theta_{I}(i))}
\tilde{h}_{j}
\Big)
\otimes
f_{i}
}
{
 \Delta g_{\theta_{I}(i)} \big \rangle
}\\
=&
\sum_{M \subset N}
\sum_{K \subset J}
\sum_{
\substack{
\pi : J \setminus K \rightarrow M \\  \textnormal{biject.}}}
\sum_{
\substack{
\theta_{I}: I \rightarrow N \setminus M \\  \textnormal{biject.}}}
\sum_{\theta: K \rightarrow I}
\Big(
\prod_{j \in J \setminus K}
\big \langle
\tilde{h}_{j},
[g_{\pi(m)}]_{\bar{\mu}_{\pi(m)}} 
 \big \rangle
\Big)\\
{}&\quad \times
\Big(
\prod_{i  \in I}
1\{\mu_{i} = \bar{\mu}_{\theta_{I}(i)}\}
\Bscal{
\Big(
\prod_{j \in \theta^{-1}(i)}
\tilde{h}_{j}
\Big)
\otimes
f_{i}
}
{
 \Delta g_{\theta_{I}(i)} 
 }
\Big)\\
=&
\sum_{K \subset J}
\sum_{\sigma: (J \setminus K) \sqcup I \rightarrow N
\atop
\textnormal{biject.}}
\sum_{ \theta: K \rightarrow I}
\Big(
\prod_{j \in J \setminus K}
\bscal{ 
\tilde{h}_{j}}
{ [g_{\sigma(j)}]_{\bar{\mu}_{\sigma(j)}}}
\Big)\\
{}&
\quad \times
\Big(
\prod_{i \in I}
1\{
\mu_i = \bar{\mu}_{\sigma(j)}\}
\Bscal{
\Big(
\prod_{k \in \theta^{-1}(i)}
\tilde{h}_{k}
\Big)
\otimes  f_{i}
}{ \Delta g_{\sigma(i)}
}
\Big)\;.
\end{split}
\end{equation*}
In the second equality above we used Lemma~\ref{lemma_tensor_identity}. In the third equality we above we used the fact that the line before vanishes unless $\theta_{I}$ is a bijection from $I$ to $N \setminus M$. 

In the fourth equality we used that the line before vanishes unless $|\theta_{J}^{-1}(m)| = 1$ for every $m \in M$. 
Therefore the sum over the map $\theta_{J}$ in the line before can be written as a sum over $K = \theta_{J}^{-1}(N \setminus M)$, $\pi = \theta_{J}\restriction_{J \setminus K}$ and $\theta = (\theta_{J} \restriction_{K}) \circ \theta_{I}^{-1}$.
In the last equality we exchange the sums over $\pi$ and $\theta_{I}$, and $M$ for just a single sum over $\sigma$.  
\end{proof}
\begin{remark}
The result of Theorem~\ref{thm:graft_adjoint} was actually first obtained in \cite[Proposition 4.4]{hoffman}, and reflects the duality between the Grossman-Larson and Connes-Kreimer Hopf algebras.  
\end{remark}

\subsection{Coherence and grafting}\label{subsec:coherence}

The next lemma states a generalisation of \eqref{eq:upsilon_def2} and shows that maps $\Upsilon_{\mu}$ interact well with our grafting procedure.  
\begin{lemme}\label{lemma:grafting_upsilon}
Given any $\tilde{f}, f \in \CF$ with $|\tilde{f}| + |f| \le N-1$, and $\mu \in [\dimX]$ one has  
\[
\Upsilon_{\mu} [\tilde{f} \graft f]
= D^{\num \tilde{f}} \Upsilon_{\mu}[f] 
\big[ 
\big(
\Upsilon[\tilde{h}] 
\big)_{\tilde{h} \in \tilde{f}}
\big]\;.
\]
\end{lemme}
\begin{proof}
We prove the general case by induction in $|f|$ - the base case when $f = \ind$ is immediate.

We write $f$ and $\tilde{f}$ as in \eqref{eq: productform}. 
Now we note that, by Leibnitz rule and the definition of $\Upsilon$, for any $v = (v_{j})_{j \in J} \in (\R^{\dimY})^{J}$, 
\begin{equation}\label{eq:Leibnitz_rule}
\begin{split}
{}
&D^{\num \tilde{f}}\Upsilon_{\mu}[f][v]\\
{}& \quad = 
\sum_{ 
(J_{i}: i \in I)}
  D^{\num f + \num \tilde{f} - \sum_{i \in I} |J_{i}|}
\sigma_{\mu}
\Big[\big( 
D^{|J_{i}|}\Upsilon_{\mu_{i}}[f_{i}][v_{J_{i}}]\big)_{ i \in I }
\sqcup v_{J^{c}} \Big]\;,
\end{split}
\end{equation} 
where above the sum is over collections of disjoint subsets $(J_{i}: i \in I)$, $J_{i} \subset J$ and we write $v_{J_{i}} = (v_{j})_{j \in J_{i}}$ and $J^{c} = J \setminus \big( \bigcup_{i \in I} J_{i} \big)$.  

On the other hand, using the definition of the grafting procedure followed by that of $\Upsilon_{\mu}$
\begin{equation*}
\begin{split}	
{}&\Upsilon_{\mu} [\tilde{f} \graft f]\\
=&
\sum_{K \subset J}
\sum_{ \theta: K \rightarrow I}
\Upsilon_{\mu}
\Big[
\Big(
\prod_{j \in J \setminus K}
\tilde{h}_{j}
\Big)
\Big(
\prod_{i \in I}
\Big[
\big(
\prod_{k \in \theta^{-1}(i)}
\tilde{h}_{k}
\big)
\graft f_{i}
\Big]_{\mu_i}
\Big) 
\Big]\\
=&
\sum_{K \subset J}
\sum_{ \theta: K \rightarrow I}
 D^{|J \setminus K| + \num f}\sigma_{\mu}
\Bigg[
\Big(
\Upsilon[
\tilde{h}_{j}]
\Big)_{j \in J \setminus K}
\sqcup
\Big(
\Upsilon_{\mu_{i}}
\Big[
\big(
\prod_{k \in \theta^{-1}(i)}
\tilde{h}_{k}
\big)
\graft f_{i}
\Big]
\Big)_{i \in I} 
\Bigg]\;.
\end{split}
\end{equation*}
Now, by our induction hypothesis, for any  $\theta:K \rightarrow I$ and $i \in I$ we have
\[
\Upsilon_{\mu_{i}}
\Big[
\big(
\prod_{k \in \theta^{-1}(i)}
\tilde{h}_{k}
\big)
\graft f_{i}
\Big]
=
D^{ |\theta^{-1}(i)|}
\Upsilon_{\mu_{i}}[f_{i}]
\Big[
\big(  
\Upsilon[\tilde{h}_{k}] \big)_{k \in \theta^{-1}(i)}
\Big]\;.
\]
Inserting the above calculation into the previous equation gives 
\begin{equation*}
\begin{split}
{}&\Upsilon_{\mu} [\tilde{f} \graft f] \\
{}&=
\sum_{K \subset J}
\sum_{ \theta: K \rightarrow I}
 D^{|J \setminus K| + \num f}\sigma_{\mu}
\Big[
\big(
D^{|\theta^{-1}(i)|}
\Upsilon_{\mu_{i}}[f_{i}]
[(\Upsilon[\tilde{h}_j])_{j \in \theta^{-1}(i)}]
\big)_{i \in I} 
\sqcup
(\Upsilon[\tilde{h}_j])_{j \in J \setminus K}
\Big]\;,
\end{split}
\end{equation*}
which finishes the proof by comparing with \eqref{eq:Leibnitz_rule}, identifying $J_{i} \leftrightarrow \theta^{-1}(i)$, and setting $v = (\Upsilon[\tilde{h}_j])_{j \in J}$. 
\end{proof}
In what follows we define, for $g \in \forest{N-1}$ and $\mu \in [d]$, $\bUpsilon_{\mu}[g](\cdot) = \frac{1}{S(g)}\Upsilon_{\mu}[g](\cdot) $ and, for $h \in \rtree{ N-1}$, $\bUpsilon[h](\cdot)  = \frac{1}{S(h)}\Upsilon[h](\cdot) $. 
We also write 
\[
\bUpsilon_{\mu}(\cdot) = \sum_{g \in\forest{N-1}} \bUpsilon_{\mu}[g](\cdot)g\;. 
\]
In particular, $\bUpsilon_{\mu}: \R^{\dimY} \rightarrow \Span(\forest{N-1})^{\dimY}$ 
and for any $g \in\forest{N-1}$ we have $\bscal{\bUpsilon_{\mu}}{g} = \Upsilon_{\mu}[g]$. 
We also define
\[
\bUpsilon(\cdot)
=
\sum_{h \in \rtree{N-1}} \bUpsilon[h](\cdot)h\;.  \]
We see that if a tree path $\Y$ is coherent then, for any $s \in (0,1]$, we have $\delta\Y_{s} = \bUpsilon(Y_{s})$ and, for every $\mu \in [\dimX]$,  $\sigma_{\mu}(\Y_{s}) = \bUpsilon_{\mu}(Y_{s})$. 
\begin{lemme}\label{lemma:coherence}
Given any $ f \in\forest{N-1}$  and $\mu \in [\dimX]$, 
\[
\sum_{\tilde{f} \in\forest{N-1}}
\scal{ \tilde{f} \otimes f} { \Delta \bUpsilon_{\mu}} \tilde{f}
=
P_{\le N - |f| - 1}
\sum_{p =0}^{N-1}
\frac{1}{p!} D^{p}\bUpsilon_{\mu}[f]
\Big[
\underbrace{(\bUpsilon,\dots,\bUpsilon)}_{p\textnormal{-times}}\Big]\;.
\]
\end{lemme}
\begin{proof}
We observe that 
\begin{equation*}
\begin{split}
\scal{ \tilde{f} \otimes f} { \Delta \bUpsilon_{\mu}} \tilde{f}
=& \frac{1}{S(\tilde{f})S(f)}
\bscal{ \tilde{f} \otimes f} { \Delta \bUpsilon_{\mu}} \tilde{f}\\
=&
\frac{1}{S(\tilde{f})S(f)}
\bscal{ \tilde{f} \graft f} { \bUpsilon_{\mu}} \tilde{f}
=
\frac{1\{ |f| + |\tilde{f}| < N\}}{ S(\tilde{f}) S(f)} 
 \Upsilon_{\mu}[ \tilde{f} \graft f] \tilde{f}\\
=& 
\frac{1\{ |f| + |\tilde{f}| < N\}}{ S(\tilde{f}) S(f)} 
 D^{\num \tilde{f}} \Upsilon_{\mu}[f]
 \Big[
\big(
\Upsilon[\tilde{h}]
\big)_{\tilde{h} \in \tilde{f}}
\Big] \tilde{f} \\
=& 
1\{ |f| + |\tilde{f}| < N\}
\frac{\mathrm{Multi}(f) }{\num \tilde{f}!}  D^{\num \tilde{f}} \bUpsilon_{\mu}[f]
\Big[
\big(
\bUpsilon[\tilde{h}] \tilde{h}
\big)_{\tilde{h} \in \tilde{f}}
\Big] \;.
\end{split}
\end{equation*} 
 
In the first equality above we are using  \eqref{eq:inner_product_relation} and in the second equality we are applying Theorem~\ref{thm:graft_adjoint}. 
In the third equality above we are using Lemma~\ref{lemma:grafting_upsilon}, and for the third equality we define the multinomial coefficient
\[
\mathrm{Multi}(\tilde{f}) =
\frac{\num \tilde{f}! }{\prod_{\tilde{h} \in \rtree{N-1}} \tilde{f}_{\tilde{h}}!}
\] 
where $\tilde{f}_{\tilde{h}}$ in again the number of instances of $\tilde{h}$ in $\tilde{f}$, so that we have
\[ S(\tilde{f}) = 
\frac{\num \tilde{f}!}{\mathrm{Multi}(\tilde{f})}  \prod_{\tilde{h} \in \tilde{f}}S(\tilde{h})\;.
\] 
It follows that
\begin{equation}\label{eq:identification_of_multilinear_deriv}
\begin{split}
\sum_{\tilde{f} \in\forest{N-1}}
\scal{ \tilde{f} \otimes f} { \Delta \bUpsilon_{\mu}} \tilde{f}
=& 
\sum_{\tilde{f} \in\forest{N-1}
\atop
|\tilde{f}| < N - |f|}
\frac{\mathrm{Multi}(\tilde{f})}{\num \tilde{f}!}   D^{\num \tilde{f}} \bUpsilon_{\mu}[f]
\Big[
\big(
\bUpsilon[\tilde{h}] \tilde{h}
\big)_{\tilde{h} \in \tilde{f}}
\Big]\\
=&
\sum_{p =0}^{N-1}
\frac{1}{p!}
\sum_{\tilde{f} \in\forest{N-1}
\atop
|\tilde{f}| < N - |f|}
1\{ \num f = p\} 
\mathrm{Multi}(\tilde{f}) D^{p}\bUpsilon_{\mu}[f]
\Big[
\big(
\bUpsilon[\tilde{h}]
\big)_{\tilde{h} \in \tilde{f}}
\Big]\\
=&
\sum_{p =0}^{N-1}
\frac{1}{p!}
\sum_{h_{1},\dots,h_{p} \in \rtree{N-1}} D^{p}\bUpsilon_{\mu}[f]
\Big[
\big(
\bUpsilon[h_{i}] h_{i}
\big)_{i=1}^{p}
\Big]\\
=&
P_{\le N - |f| - 1}
\sum_{p =0}^{N-1}
\frac{1}{p!}  D^{p}\bUpsilon_{\mu}[f]
\Big[
\underbrace{(\bUpsilon,\dots,\bUpsilon)}_{p\textnormal{-times}}\Big]\;.
\end{split}
\end{equation}
\end{proof}
The above lemma gives us the following corollary, the main result of this section. 
It expresses remainders at higher levels in terms of remainders at level 0.
\begin{corrolaire}\label{cor:niceformula}
For any coherent tree path $\Y$, $f \in\forest{N-1}$, and $\mu \in [\dimX]$, we have
\begin{equation}\label{eq:rem_formula-2}
\begin{split}
R^{\sigma_\mu(\Y),f}_{s,t}
&=
\bUpsilon_{\mu}[ f ](Y_{t})
-
\Big\langle\X_{s,t},
P_{\le N - |f| - 1}
\sum_{p =0}^{N-1}
\frac{1}{p!} D^{p}\bUpsilon_{\mu}[f](Y_{s})
\Big[
\underbrace{(\bUpsilon(Y_{s}),\dots,\bUpsilon(Y_{s}))}_{p\textnormal{-times}}\Big] 
\Big\rangle\;.
\end{split}
\end{equation}
\end{corrolaire}
\begin{proof}
Using the definition of coherence we can write
\begin{equation*}
\begin{split}
R^{\sigma_{\mu}(\Y),f}_{s,t} 
=& 
\scal{f}{ \sigma_{\mu}(\Y_t)} - \scal{ \X_{s,t} \otimes f}{ \Delta \sigma_{\mu}(\Y_t)}\\
=&
\bUpsilon_{\mu}[f](Y_{t})
-
\scal{ \X_{s,t} \otimes f}{ \Delta \bUpsilon_{\mu}(Y_s)}.	
\end{split}
\end{equation*}
We then use Lemma~\ref{lemma:coherence} to see that
\begin{equation*}
\begin{split}
\scal{ \X_{s,t} \otimes f}{ \Delta \bUpsilon_{\mu}(Y_s)}
=& 
\Big\langle 
\X_{s,t},
\sum_{\tilde{f} \in\forest{N-1}}
\scal{\tilde{f} \otimes f}
{\Delta \bUpsilon_{\mu}(Y_s)} \tilde{f}
\Big\rangle\\
=& 
\Big\langle 
\X_{s,t},
P_{\le N - |f| - 1}
\sum_{p =0}^{N-1}
\frac{1}{p!} D^{p}\bUpsilon_{\mu}[f](Y_{s})
\Big[
\underbrace{(\bUpsilon(Y_{s}),\dots,\bUpsilon(Y_{s}))}_{p\textnormal{-times}}\Big]
\Big\rangle\;.
\end{split}
\end{equation*}
\end{proof}
\begin{remark}
Note that if we know that the functions $(\sigma_{\mu})_{\mu=1}^{d}$ are actually smooth then the $\Upsilon_{\mu}$  make sense on $\CF_{\all}$  and it easy to check that the various lemmas and corollaries of this subsection hold without any constraint on the orders of the participating forests.
\end{remark}

\subsection{The main remainder estimate}
In order to estimate \eqref{eq:rem_formula-2} we will use the Taylor expansion/remainder formula of \cite[Proposition A.1]{hairer2014theory}.

We will want to impose a stopping rule on generating our Taylor expansion that takes into the consideration the order forests that are produced, a Taylor expansion like \cite[Proposition A.1]{hairer2014theory} is convenient because it lets us enforce a rule on how the expansion is iteratively generated.

We say $A \subset \CF$ is \emph{full} if it is non-empty, finite, and, for any $f \in A$, we have 
\[
\{\tilde{f}  \in \mathcal{F}: \tilde{f}_{h} \le f_{h} \textnormal{ for every } h \in \CF \}
\subset A\;.
\]
Note that, for any $M \in \mathbb{N}$, $\forest{M}$ is full. 

We now fix some arbitrary total order on  $\mathring{\mathcal{T}}_{\le N-1}$. 
For any $A \subset \CF$ we define a ``boundary'' of $A$ by setting
\begin{equation}\label{def-partialA}
\partial A = 
\left\{ f \in \CF \setminus (A \cup \ind) : 
(f_{h}  - 1\{h = \mathfrak{m}(f)\} : h \in \rtree{N-1}) \in A
\right\}\;,
\end{equation} 
where 
\[
\mathfrak{m}(f) = \min \{h \in \rtree{N-1} : f_{h} \not = 0\}\;
\]  
is the minimal tree appearing in the non-empty forest $f$. 
In words, $f \in \partial A$ if $f$ is non-empty, does not belong to $A$, but if one reduces the number of instances of the minimal tree appearing in $f$ by $1$ then the resulting forest does belong to $A$. %

We can then state the formulation of \cite[Proposition A.1]{hairer2014theory} we will use. 
\begin{proposition}\label{prop:tay_remainder}
Let $A \subset \CF$ be full and $F \in  C^{K + 1}([0,1]^{\rtree{N-1}})$ where $K = \max_{f \in A} \num f$, then
\[
F(\underline{1}) =
\sum_{f \in A}
\frac{\partial^{f}F(0)}{f!}
+
\sum_{f \in \partial A}
\int_{[0,1]^{\rtree{N-1}}} 
\partial^{f}F(y)
Q^{f}(dy)
\]
where $\underline{1} = (1 : h \in \CF) \in  [0,1]^{\rtree{N-1}}$ is vector of all ones and we are using standard notation for derivatives and monomials indexed by multi-indices $f = (f_{h} : h \in \rtree{N-1}) \in \CF$, and for each $f \in \CF$, $Q^{f}(dy)$ is measure of total mass $1$ on $[0,1]^{\rtree{N-1}}$. 
\end{proposition}
We can now state the key lemma for this section. 
\begin{corrolaire}
\label{cor:one_controls_others} 
Suppose the tree path $\Y$ is coherent, then, for any $\mu \in [\dimX]$ and $f \in\forest{N-1}$,  
\begin{equation}
\begin{split}
\label{e:wish}
&
|R^{\sigma_\mu (\Y),f}_{s,t}|
\lesssim
\sup \left\{ 
\big|D \Upsilon_{\mu}[f]\big( Y_{s} + a \big) \big|  \cdot  |R^{\Y,\ind}_{s,t}|
:
a \in \mathbb{R}^{\dimY}, |a| \le E_{s,t}^{\Y} + |Y_{t}-Y_{s}|
\right\}\\
{}& +
\sup
\left\{ 
\Big|  D^{\num \tilde{f} } \Upsilon_{\mu}[f](Y_{s} + z) \prod_{h \in \tilde{f}} \Upsilon[h](Y_{s})\scal{\X_{s,t}}{h} \Big|:\ 
\begin{array}{c}
 \tilde{f} \in \partial \forest{N- |f| -1} \\
z \in \R^{\dimY}, |z| \le E_{s,t}^{\Y}
\end{array}
\right\}\;,
\end{split}
\end{equation}
where 
\begin{equation}\label{eq:expansion_bound}
E_{s,t}^{\Y} = \sum_{h \in \rtree{N-1}} | \Upsilon[h](Y_{s}) \scal{\X_{s,t}}{h}|\;.
\end{equation} 
\end{corrolaire}
Before turning to the proof, we point out that \eqref{e:wish} estimates $R^{\sigma_\mu (\Y),f}_{s,t}$ by treating it as a multivariable Taylor remainder - this is why we see suprema of derivatives on the right hand side above. 
Being careful about the constraints on the arguments of these derivatives will be important, in particular when working with coefficients $\sigma$ that are allowed to grow at infinity,  see Assumption~\ref{Ass-Sigma-Polynomial}.  
\begin{proof}
Given any $x = (x_{h}) \in [0,1]^{\rtree{N-1}}$ we define
\[
\Y_{s,t}(x)
=  \sum_{h \in \rtree{N-1}}  x_{h} \bUpsilon[h](Y_{s}) \scal{\X_{s,t}}{h}\;.
\]
We then have
\begin{equation*}
Y_{s} + \Y_{s,t}(\underline{1}) + R^{\Y,\ind}_{s,t} = Y_{t}
\end{equation*}

We now define the function $F: [0,1]^{\rtree{N-1}} \rightarrow \R^{\dimY}$ given by   
\[
F(x) = \bUpsilon_{\mu}[f]
\Big( Y_{s} + \Y_{s,t}(x) \Big).
\]
Now, by Corollary~\ref{cor:niceformula} and adding and subtracting $F(\underline{1})$, we have 
\begin{equation}\label{eq:work_for_rembound}
\begin{split}
R^{\sigma_\mu(\Y),f}_{s,t}
&\le 
\Big| 
\bUpsilon_{\mu}[ f ](Y_{s} + \Y_{s,t}(\underline{1}) + R^{\Y,\ind}_{s,t}) - 
\bUpsilon_{\mu}[f](Y_{s} + \Y_{s,t}(\underline{1})) \Big|\\
& \quad
+
\Big| 
F(\underline{1}) 
-
\Big\langle\X_{s,t},
P_{\le  N - |f| - 1}
\sum_{p =0}^{N-1}
\frac{1}{p!} D^{p}\bUpsilon_{\mu}[f]
\Big[
\underbrace{(\bUpsilon,\dots,\bUpsilon)}_{p\textnormal{-times}}\Big] 
\Big\rangle \Big| \;.
\end{split}
\end{equation}
By the mean value theorem the first term on the right hand side of \eqref{eq:work_for_rembound} is bounded by
\[
\sup_{\lambda \in [0,1]}
\big| D\bUpsilon_{\mu}[ f ](Y_{s} + \Y_{s,t}(\underline{1}) + \lambda R^{\Y,\ind}_{s,t}) \big|
\cdot
|R^{\Y,\ind}_{s,t}|\;,
\]
which in turn is bounded by the first term on the right hand side of \eqref{e:wish}. 

We now claim that
\begin{equation}\label{eq:Tay_expansion}
\begin{split}
\sum_{\tilde{f} \in \forest{N - |f| = 1}}
\frac{\partial^{\tilde{f} }F(0) }{\tilde{f}!}
=&
\Big\langle 
\X_{s,t},
P_{\le N - |f| - 1}
\sum_{p =0}^{N-1}
\frac{1}{p!} D^{p}\bUpsilon_{\mu}[f](Y_{s})
\Big[
\underbrace{(\bUpsilon(Y_{s}),\dots,\bUpsilon(Y_{s}))}_{p\textnormal{-times}}\Big]
\Big\rangle\;\;.
\end{split}
\end{equation}
To verify that \eqref{eq:Tay_expansion} holds, note that we can substitute into \eqref{eq:Tay_expansion}  the right hand side of the first line of \eqref{eq:identification_of_multilinear_deriv} and also use that
\[
\frac{\mathrm{Multi}(\tilde{f})}{\num \tilde{f}!} = \frac{1}{\tilde{f}!}
\quad
\textnormal{and}
\quad
\Big\langle 
\X_{s,t},
 D^{\num \tilde{f}} \bUpsilon_{\mu}[f](Y_{s})
\Big[
\big(
\bUpsilon[\tilde{h}](Y_{s})\tilde{h}
\big)_{\tilde{h} \in \tilde{f}}
\Big]
\Big\rangle
=
\partial^{\tilde{f}}F(0)
\;.
\]
Applying Proposition~\ref{prop:tay_remainder} to $F(\underline{1})$ with $A=\forest{N-1-|f|}$ then gives us 
\begin{equation}\label{eq:remainder_of_remainder}
\begin{split}
{}&F(\underline{1})
-
\Big\langle 
\X_{s,t},
P_{\le N - |f| - 1}
\sum_{p =0}^{N-1}
\frac{1}{p!} D^{p}\bUpsilon_{\mu}[f](Y_{s})
\Big[
\underbrace{(\bUpsilon(Y_{s}),\dots,\bUpsilon(Y_{s}))}_{p\textnormal{-times}}\Big]
\Big\rangle \\
{}&=
\sum_{\tilde{f} \in \partial \forest{N-1-|f|} }
\int_{[0,1]^{\rtree{N-1}}}(\partial^{\tilde{f}}F)(y) Q^{\tilde{f}}(dy)\;.
\end{split}
\end{equation}
We then see the second term on the right hand side of \eqref{eq:work_for_rembound} is bounded by the second term on the right hand side of \eqref{e:wish}. 
\end{proof}

\section{ Proof of main Theorems}\label{sec:final}
\label{equ}

Before proceeding to the details of the proof we give a brief overview. The key idea is contained in the \emph{interior regularity estimate}, Lemma~\ref{interior_regularity},
which in turn builds on the Sewing Lemma, Proposition~\ref{prop:integral_bound}, and Corollary~\ref{cor:one_controls_others}. More precisely, the Sewing Lemma provides an estimate of order 
$(N+1) \alpha$ on the rough integral $Z$ in terms of the various order bounds $[\X: [f]_{\mu}]$ and the semi-norms $[\sigma_\mu(\Y) : f ]$ which measure the quality of local approximations of 
all coefficients. 
According to Corollary~\ref{cor:one_controls_others} the semi-norms $[\sigma_\mu(\Y) : f ]$ can   be controlled in terms of
\begin{itemize}
\item the size of the coefficients $\Upsilon$ on the relevant interval, 
\item the order bounds $[\X: [f]_{\mu}]$, as well as
\item the remainder $R^{\Y,\ind}_{s,t}$.
\end{itemize}
It is important to note the relevant estimate \eqref{e:wish} is \emph{linear} in $|R^{\Y,\ind}_{s,t}|$ and that $|R^{\Y,\ind}_{s,t}|$ describes an error of order $N\alpha$ in the local description of $Y$. Furthermore, the quantity $|R^{\Y,\ind}_{s,t}|$ 
 can be easily estimated in terms of the corresponding quantity for $Z$. Combining all of these estimates one obtains an estimate on the error of the local approximation of order $N+1$ of $Z$ in terms of (the coefficients $\Upsilon$, the order bounds $[\X: [f]_{\mu}]$ and) the quantities $|R^{\Z,\ind}_{s,t}|$ which measure the error in the approximation of order $N\alpha$. This can be seen as a regularity improvement from $N\alpha$ to $(N+1)\alpha$,  which in turn allows to absorb the term $|R^{\Z,\ind}_{s,t}|$ provided one works on a small-enough interval. 
The proof of Theorem~\ref{thm:main} then mostly consists of plugging (a truncated version of) the \emph{interior regularity estimate}, Lemma~\ref{interior_regularity}, into Lemma~\ref{lemma:MW}, which captures the damping of the non-linear term as  discussed in the introduction, and choosing appropriate parameters depending on the growth assumption on the coefficient $\sigma$. The final iteration argument (see Step 3 of the Proof of Theorem~\ref{thm:main} 1)) is an adaptation of the \emph{additive noise case}, \cite[Proof of Lemma 2.1, Step 6]{MW}.

We start the argument with the following simple lemma that translates Assumptions~\ref{Ass-Sigma-Bounded} and~\ref{Ass-Sigma-Polynomial} into estimates on $\Upsilon[h]$ for $h \in \rCT_{\le N}$. 
\begin{lemme}
\label{l:Assumption-imply-Upsilon}  
Let $h \in \rCT_{\le N}$ and $0 \le p \le N - |h| + 1$. 

Suppose that Assumption~\ref{Ass-Sigma-Bounded} holds, then  
\begin{equation}\label{upsilon-bound-sigma-bounded}
\sup_{y \in \R^{\dimY}}
|D^{p} \Upsilon[h](y)|
\lesssim 
C_{\sigma}^{|h|}\;.
\end{equation}

Suppose instead that Assumption~\ref{Ass-Sigma-Polynomial} holds, then 
\begin{equation}\label{upsilon-bound-sigma-poly}
|D^{p} \Upsilon[h](y)| \lesssim C_{\sigma}^{|h|} \jbrac{y}^{(\gamma - 1) |h| + 1 - p} \;.
\end{equation}
\end{lemme}
\begin{proof}
We treat the case $p = 0$, the case of $p \not = 0$ can be treated similarly - the constraint on $p$ comes from the fact that we control only $N$ derivatives of the $(\sigma_{\mu})_{\mu \in \dimX}$. 

A straightforward inductive argument shows that 
\[
|\Upsilon[h](y)|
\lesssim 
\prod_{u \in N_{h}}
|D^{c(u)} \sigma_{\mu(u)}(y)|\;,
\]
where on the right hand side we are taking an appropriate operator norm, $N_{h}$ denotes the set of nodes of the tree $h$ and, given $u \in N_{h}$, we write $c(u)$ for the number of children nodes of $u$ in $h$ (that is, the number of nodes directly attached to $u$ that are further from the root) and  $\mu(u) \in [d]$ for the value of the node decoration that $u$ carries. 

Using the fact that the cardinality of $N_{h}$ is $|h|$, the bound in the case of Assumption~\ref{Ass-Sigma-Bounded} follows immediately. 

In the case of Assumption~\ref{Ass-Sigma-Polynomial}, the first bound follows from the fact that $\sum_{u \in N_{h}} c(u)$ is just the number of edges of $h$ which is $|h|-1$.  
\end{proof}
We can alternatively formulate the above lemma by saying that, for any $f \in\forest{N-1}$ and $\mu \in [d]$,  we have, uniformly over $y \in \mathbb{R}^{\dimY}$ and $0 \le p \le N$, 
\begin{equation*}
|D^{p}\Upsilon_{\mu}[f](y)|
\lesssim
\begin{cases}
C_{\sigma}^{|f|+1} 
&  
\text{ under the Assumption~\ref{Ass-Sigma-Bounded},}\\
C_{\sigma}^{|f|+1} \jbrac{y}^{(\gamma -1) (|f| +1) +1 - p}
&
\text{ under Assumption~\ref{Ass-Sigma-Polynomial}.}
\end{cases}
\end{equation*}
We are now ready to pass to the proofs of Theorems~\ref{thm:main} and  Theorems~\ref{thm:main2}.  
The following ``interior regularity estimate'' which gives a control of the $C^\alpha$ norm of $Z$ and $Y$ in terms of its lower regularity $L^\infty$ norm is 
key to both arguments.
To keep the length of some formulas within reason we use the following shorthand notation combining some of the quantities appearing on the right hand side of \eqref{e:wish}
\begin{equation}\label{e:DefUf}
\begin{split}
U_{\mu}(f, \Y,I ) &\eqdef  
\sup \big\{
 \big| D \Upsilon_{\mu}[f ] ( 
Y_{s} + a ) \big|: 
s,t \in I,  |a| \le E_{s,t}^{\Y} + |Y_{t}-Y_{s}| \big\}
   \; , \\
U_{\mu}(f,\bar{f}, \Y,I) &\eqdef   
\sup \bigg\{
\big|
D^{ \num \bar{f}} \Upsilon_{\mu}[f](Y_{s} + z) \prod_{h \in \bar{f}} \Upsilon[h](Y_{s}) 
	\big|: 
\begin{array}{c}
s, t \in I, \\
|z| \le E_{s,t}^{\Y} 
\end{array}
\bigg\} \; . 
\end{split}
\end{equation}

Below, for any $j$, we denote by $\CF_{= j}$ the set $\{ f \in \CF: |f| = j\}$. 

\begin{lemme}[Interior regularity]
\label{interior_regularity}
Let $\Y$ solve \eqref{RDE3} on $ [0,1]$  in the sense of Definition~\ref{def:solution}  and let $\Z$  be given by \eqref{eq:integral as a path}. Let $I \subseteq [0,1]$ be a closed interval of length $L$
and assume that $L$ is small enough to ensure that for all forests $f \in\forest{N-1}$ and all $\mu \in [\dimX]$ 
\begin{align}
\label{e:step2-condition}
L^{ (|f|+1) \alpha} \left[\X:[f]_\mu \right]  	U_{\mu}(f, \Y,I ) 
  \leq \epsilon \;,
\end{align}
for a suitably small  $\epsilon = \epsilon(\alpha,d,k)\leq 1$. 
 Then we have the estimates 
\begin{align}
\notag
L^{N\alpha}[\Z \colon \ind]_I  
&\lesssim
	\sum_{f \in\forest{N-1}  }\sum_{\mu\in[\dimX]}
		[\X:[f]_\mu]  
			\\
			\notag
		& \quad   \cdot \Bigg \{
		 	 L^{ (|f| +1) \alpha +1 }   
			U_{\mu}(f, \Y,I )
			 \| Y \|_I^m		
			+
			\max_{\bar{f}  \in \partial \forest{N- |f| -1}  }  
 				L^{ (|\bar{f} | +|f| +1 ) \alpha  } 
 				U_{\mu}(f,\bar{f}, \Y,I)
				[\X : \bar{f} ]
		 \Bigg\} 	\\		
\label{e:step3-final1bis}
&
+
\sum_{f \in \CF_{=N-1}}
\sum_{\mu \in [\dimX]}
L^{N\alpha}
[\X:[f]_\mu]   \|  \Upsilon_{\mu}[f](Y_{\bullet})   \|_I \,.
\end{align}
\end{lemme}

\begin{proof}

\noindent{\bf Part 1: Application of Sewing Lemma.} 
 For $0 \leq s \leq t \leq 1$, the estimate \eqref{eq:roughintegral_bound} implies that
\begin{equation*} %\label{eq:Gubinelli_bound}
\Big|Z_t-Z_s
	- \Xi^{\sigma(\Y)}_{s,t}
\Big| \lesssim 
 |t-s|^{(N+1)\alpha}
\sum_{f  \in\forest{N-1}}\sum_{\mu \in [\dimX] }
\left[\X:[f]_\mu \right]
\left[\sigma_\mu(\Y) : f\right]_{[s,t]}\;,
\end{equation*}
where 
\[
\Xi^{\sigma(\Y)}_{s,t} 
\eqdef
\sum_{\mu \in [\dimX]}
\Xi^{\sigma_{\mu}(\Y)}_{s,t} 
=
\sum_{\mu \in [\dimX]}
\sum_{f \in\forest{N-1}}
\scal{\X_{s,t}}{[f]_{\mu}} 
\scal{f}{\sigma_{\mu}(\Y)_{s}}
 \in \R^{\dimY}\;.
\]
By coherence  of  $\Y$,  i.e. by  \eqref{eq:lhs_to_rhs}, we get, 
\begin{align*}
\Xi^{\sigma(\Y)}_{s,t}
&=
\sum_{h \in \rtree{N-1}}
\scal{\X_{s,t}}{h}\scal{h}{\Y_s}
+
\sum_{ 
f \in \CF_{=N-1}}
\sum_{\mu \in [\dimX]}
\scal{\X_{s,t}}{[f]_{\mu}}\scal{f}{\sigma_{\mu}(\Y)_s} \;.
%\label{e:Xi-coherence}
\end{align*}
Recalling that $\delta \Y = \delta \Z$ this leads to  
\begin{align} \notag
| R^{\Z,\ind}_{s,t}|  
&= 
\Big|Z_t-Z_s
	- \scal{\X_{s,t}}{\delta \Y_{s}}
\Big|\\ \notag
{}& \le \Big|
Z_t-Z_s
	-  \Xi^{\sigma(\Y)}_{s,t}
\Big| +
\Big|
 \Xi^{\sigma(\Y)}_{s,t} - 
 \scal{\X_{s,t}}{\delta \Y_{s}}
\Big| \\ \notag
& \lesssim 
  |t-s|^{(N+1)\alpha}
\sum_{f \in\forest{N-1}} \sum_{\mu \in [\dimX]}
\left[\X:[f]_\mu\right]
\big[ \sigma_\mu(\Y) : f\big]_{[s,t]}
\\
\notag
&
+
|t-s|^{N \alpha}
\sum_{ \CF_{ = N-1 }}
\sum_{\mu \in [\dimX]}
[\X:[f]_\mu]  \cdot  | \scal{f}{\sigma_{\mu}(\Y)_s} | \;.
\end{align} 
Dividing by $|t-s|^{N\alpha}$ and taking the supremum over all $s<t$ in $I$ we arrive at 
\begin{align} \notag
[\Z \colon \ind]_I   & \lesssim   
L^\alpha
\sum_{f \in\forest{N-1} }\sum_{\mu \in [\dimX]}
\left[\X:[f]_\mu \right] \big[ \sigma_\mu(\Y) : f\big]_{I}
\\
\label{eq:22}
&
+
\sum_{f \in \CF_{=N-1}}
\sum_{\mu \in [\dimX]}
[\X:[f]_\mu] \cdot    \|  \Upsilon_{\mu}[f](Y_{\bullet})   \|_I \;.
\end{align}

\noindent{\bf Part 2: Application of Corollary~\ref{cor:one_controls_others}.} 
We use the estimate \eqref{e:wish} in  Corollary~\ref{cor:one_controls_others}  in the form
\begin{equation}
\label{e:Step2-estimate1}
[  \sigma_\mu (\Y) : f ]_I
\lesssim
	L^{ |f| \alpha}   
U_{\mu}(f, \Y,I )
	[\Y : \ind ]_I  
+
\max_{
\bar{f}  \in  \partial \forest{N- |f| -1}}
 L^{ (|\bar{f} | +|f| -N) \alpha  }
U_{\mu}(f,\bar{f}, \Y,I)
[\X : \bar{f} ] \; .
\end{equation}
Note that $| \bar{f} | \geq N - |f|$, so that the exponent $(|\bar{f} | +|f| -N) \alpha $ in the second term on the RHS of \eqref{e:Step2-estimate1} is always non-negative.

We plug this estimate into the RHS of \eqref{eq:22} resulting in 
\begin{align}
 \notag
[\Z \colon \ind]_I  &  \lesssim  
	 L^\alpha
		\sum_{ f \in\forest{N-1}  }\sum_{\mu \in [\dimX]}
		[\X: [f]_\mu]  
\\ \notag	
		& \qquad   \cdot \Bigg \{
		 	 L^{ |f| \alpha}   
			U_{\mu}(f, \Y,I ) 
			[\Y : \ind ]_I  		
			+
			\max_{\bar{f}  \in  \partial \forest{N- |f| -1} }  
 				L^{ (|\bar{f} | +|f| -N) \alpha  } 
 				U_{\mu}(f,\bar{f}, \Y,I) 
				[\X : \bar{f} ]
		 \Bigg\} 
\\
\label{eq:22bis}
&
+
\sum_{ \CF_{=N-1} }
\sum_{\mu \in [\dimX]}
[\X:[f]_\mu]  \cdot  \|  \Upsilon_{\mu}[f](Y_{\bullet})   \|_I \;.
\end{align} 

The key observation at this point is that the terms $	[\Y : \ind ]_I  $ on the RHS can be replaced by $[\Z \colon \ind]_I$. Indeed, 
equation \eqref{eq:remainder_relation} which relates the tree paths $\Y$ and $\Z$ immediately implies  the bound 
 \begin{equation*}%\label{conversion-Y-Z}
 [\Y  :\ind]_I \leq   [ \Z :\ind]_I  + L^{1-N\alpha} \| Y \|_I^m \;.
 \end{equation*}
Plugging this into the RHS of \eqref{eq:22bis}, the resulting terms involving    $[ \Z :\ind]_I$  can be absorbed into the LHS, provided their pre-factors are small. 
More precisely,  
\begin{align}
 \notag
[\Z \colon \ind]_I   & \lesssim  
	 L^\alpha
		\sum_{ f \in\forest{N-1}  }
		\sum_{\mu \in [\dimX]}
		[\X:[f]_\mu ]  
\\ \notag	
		& \qquad   \cdot \Bigg \{
		 	 L^{ |f| \alpha}   
			U_{\mu}(f, \Y,I ) 
			L^{1-N\alpha} \| Y \|_I^m		
			+
			\max_{\bar{f}  \in  \partial \forest{N- |f| -1} }  
 				L^{ (|\bar{f} | +|f| -N) \alpha  } 
 				U_{\mu}(f,\bar{f}, \Y,I) 
				[\X : \bar{f} ]
		 \Bigg\} 
\\
\notag
%\label{e:step2-final}
&
+
\sum_{|f| \in  \CF_{=N-1} }
\sum_{\mu \in [\dimX]}
[\X:[f]_\mu] \cdot   \|  \Upsilon_{\mu}[f](Y_{\bullet})   \|_I \;,
\end{align} 
provided for all forests $f \in\forest{N-1}$ and all $\mu \in [\dimX]$ 
\begin{align*}
L^{ (|f|+1) \alpha} \left[\X:[f]_\mu \right]  	U_{\mu}(f, \Y,I ) 
  \leq \epsilon \;,
\end{align*}
for a suitably small  $\epsilon = \epsilon(\alpha,d,k)\leq 1$. This  is the desired estimate \eqref{e:step3-final1}. 
\end{proof}
It turns out, that for the proof of Theorems~\ref{thm:main} and~\ref{thm:main2} only the following truncated version of the estimate \eqref{e:step3-final1bis} is required  
\begin{corrolaire}
\label{cor:interior_regularity}
Let $\Y$ solve \eqref{RDE3} on $ [0,1]$  in the sense of Definition~\ref{def:solution}  and let $\Z$  be given by \eqref{eq:integral as a path}. Let $I \subseteq [0,1]$ be a closed interval of length $L$ satisfying \eqref{e:step2-condition}.
 Then we have the estimates 
\begin{align}
\notag
L^{\alpha}[Z]_{\alpha,I}  
&\lesssim
	\sum_{f \in\forest{N-1}  }\sum_{\mu\in[\dimX]}
		[\X:[f]_\mu]  
			\\
			\notag
		& \quad   \cdot \Bigg \{
		 	 L^{ (|f| +1) \alpha +1 }   
			U_{\mu}(f, \Y,I )
			 \| Y \|_I^m		
			+
			\max_{\bar{f}  \in  \partial \forest{N- |f| -1}  }  
 				L^{ (|\bar{f} | +|f| +1 ) \alpha  } 
 				U_{\mu}(f,\bar{f}, \Y,I)
				[\X : \bar{f} ]
		 \Bigg\} 	\\		
\label{e:step3-final1}
&
+
\sum_{h \in \rtree{N}} 
\sum_{\mu \in [\dimX]}
L^{|h| \alpha}
[\X:h]   \|  \Upsilon[h](Y_{\bullet})   \|_I \,
\end{align}
and
\begin{align}
L^\alpha [Y]_{\alpha,I} \lesssim \text{ RHS of \eqref{e:step3-final1}} + L \| Y \|_I^m \;.
\label{e:step3-final2}
\end{align}

\end{corrolaire}
\begin{proof}
The estimate \eqref{e:step3-final1} follows from \eqref{e:step3-final1bis} combined with
\begin{align}
\notag
L^\alpha[Z]_{\alpha,I} 
&\le 
\sum_{h \in \rtree{N-1}} 
	L^{|h|\alpha} [\X:h] \|  \Y : h \|_I  
	+ L^{N\alpha}[\Z:\ind]_{I} \; ,
\end{align}
as well as the identification
\begin{align*}
\sum_{ \CF_{=N-1} }
\sum_{\mu \in [\dimX]} L^{N\alpha}
[\X:[f]_\mu]  \cdot  \|  \Upsilon_{\mu}[f](Y_{\bullet})   \|_I  = 
\sum_{h \in \CT_{=N} }  L^{|h| \alpha}
[\X:h]  \cdot  \|  \Upsilon[h](Y_{\bullet})   \|_I  \;.
\end{align*}
The estimate \eqref{e:step3-final2} is an immediate consequence of  the defining relationship $Y_t = y_0 - \int_0^t |Y_r|^{m-1} Y_r dr + Z_t$.
\end{proof}
\begin{proof}[Proof of Theorem~\ref{thm:main} (1) - The case of bounded coefficients.] 

\vspace{.1cm}
\noindent{ \bf Step 1: Simplification of interior regularity estimate.}
We now work under Assumption~\ref{Ass-Sigma-Bounded}, i.e. we assume  that $\sigma$ and its derivatives up to order $N$ are bounded.
Using the boundedness of coefficients, and estimate \eqref{upsilon-bound-sigma-bounded} in Lemma~\ref{l:Assumption-imply-Upsilon}, this  the estimate \eqref{e:step3-final1} simplifies to 
\begin{align}
\notag
L^\alpha[Z]_{\alpha,I} 
&\lesssim 
\sum_{h \in \rtree{N}} 
	L^{|h|\alpha} [\X:h]  
\\\notag
&+
\sum_{f \in\forest{N-1} }
\sum_{\mu \in [\dimX]}    
	 L^{ (|f| +1) \alpha +1 }
	 [\X:[f]_\mu]
	 \| Y \|_I^m		
\\
&+
	\sum_{f \in\forest{N-1} }\sum_{\mu\in[\dimX]}
			\max_{\bar{f}  \in  \partial \forest{N- |f| -1}  }    
 				L^{ (|\bar{f} | +|f| +1) \alpha  }
				[\X:[f]_\mu]   
 				[\X : \bar{f} ]
\label{e:step4-1}
\end{align}
provided $I$ is an interval of length $L$ satisfying
\begin{align}
\label{e:step4-condition1}
L^{ (|f|+1) \alpha} \left[\X:[f]_\mu\right]   
  \leq \epsilon \;,
\end{align}
for all forests $f \in\forest{N-1}$ and for a potentially smaller value of $\epsilon = \epsilon(\alpha,d,k, C_\sigma)\leq 1$. 
In order to simplify these expressions further, we make additional assumptions: first,  we fix a $t^\star \in [0,1]$
and for the remainder of this step we assume that the interval $I$ satisfies $I \subseteq  [t^\star, 1)$.
Next, we work under the assumption that for all trees $h \in \rtree{N}$  we have 
\begin{equation}\label{eq:simplifying_assump}
[\X : h]
\le \epsilon  \|Y\|_{[t^\star, 1]}^{ |h| \delta \alpha }
\;,
\end{equation}
for   $\delta =m$.
 Under \eqref{eq:simplifying_assump} the condition \eqref{e:step4-condition1} is guaranteed as soon as  
\begin{align}
\label{e:step4-condition2}
L \leq    \| Y \|_{[t^\star,1]}^{-\delta} \;.
\end{align}
 Under these assumptions the estimate \eqref{e:step4-1} 
simplifies further to 
\begin{align}
\label{e:step4-final}
L^\alpha[Z]_{\alpha,I} 
&\lesssim 
\epsilon (1+  L   \|Y \|_{[t^\star,1]}^m +1 +  1)  \lesssim \epsilon+ \epsilon L   \|Y \|_{[t^\star,1]}^m\;.
\end{align}
The estimate \eqref{e:step3-final2}  turns into
\begin{align}
\label{e:step4-final2}
L^\alpha[Y]_{\alpha,I}  
&\lesssim 
\epsilon 	+	 L   \|Y \|_{[t^\star,1]}^m  \;.
\end{align}

\noindent { \bf Step 2: Application of Lemma~\ref{lemma:MW}}
We continue working under the boundedness Assumption~\ref{Ass-Sigma-Bounded} on the coefficients,  and assuming the bound 
\eqref{eq:simplifying_assump} on the trees $[\X : h]$. 
We apply Lemma~\ref{lemma:MW} on the interval $[t_\star,1]$  and for  
\begin{align*}%\label{e-assumption-L0}
L =   \| Y \|_{[t^\star,1]}^{-\delta} = \| Y \|_{[t^\star,1]}^{-m} \;. 
\end{align*}  
We assume throughout that all times involved are contained in $[0,1]$ which implies in particular that $L \in (0,1)$ so that \eqref{e:step4-condition2} implies 
\begin{equation} \label{Y-useful-assumption}
\| Y \|_{[t^\star,1]} \geq 1.
\end{equation}
Under these assumptions, combining Lemma~\ref{lemma:MW} with \eqref{e:step4-final}, \eqref{e:step4-final2} we  obtain for $t > t_\star +L$ 
\begin{align}
\notag
|Y(t)|  
  \lesssim 
   \max \Big\{ 
	 & (t-t^\star -L)^{-\frac{1}{m-1}}; 
	 \Big(\epsilon L^{-1} + \epsilon   \|Y \|_{[t^\star,1]}^m \Big)^{\frac{1}{m}};  \\
\notag
& \qquad 	
	\Big(     \|Y \|_{[t^\star,1]}^{m-1} \Big( \epsilon	+	 L   \|Y \|_{[t^\star,1]}^m \Big)   \Big)^{\frac{1}{m}} ; 
	  \Big(  \epsilon + L   \|Y \|_{[t^\star,1]}^m 
  \Big)    \Big\}\\
  \lesssim 
   \max \Big\{ 
	 & (t-t^\star -L)^{-\frac{1}{m-1}}; 
	  \epsilon^{\frac{1}{m}}   \|Y \|_{[t^\star,1]} ;  
	    \|Y \|_{[t^\star,1]}^{\frac{m-1}{m}}  ; 1 
    \Big\} \;.
    \notag
%\label{Step5-2} 
\end{align}
We bound  using Young's inequality

\begin{align*}
 \|Y \|_{[t^\star,1]}^{\frac{m-1}{m}}   \leq  \epsilon \|Y \|_{[t^\star,1]} +  C(\epsilon, m)\;  ,
\end{align*}
and then choose first $\epsilon$ small and  absorb all constants into 
the term  $ (t-t^\star -L)^{-\frac{1}{m-1}} \geq 1$ and conclude that if  $t$ is large enough to ensure that
\begin{align}\label{t-conditionA} 
 (t-t^\star - L)^{-\frac{1}{m-1}} \leq \bar{\epsilon} \| Y \|_{[t^\star,1]},
\end{align}
for $\bar{\epsilon} =  \bar{\epsilon}(\alpha,d,k, C_\sigma)$ small enough, we get  
\begin{align}\label{thereweare}
|Y(t) | \leq \frac12  \| Y \|_{[t^\star,1]}.
\end{align}
We observe that \eqref{t-conditionA} is guaranteed as soon as 
\begin{align*}%\label{t-condition1} 
 t-t^\star \geq  \bar{\epsilon}^{-(m-1)}  \| Y \|_{[t^\star,1]}^{-(m-1)}  + \| Y \|_{[t^\star,1]}^{-m}   .
\end{align*}
which, taking into account \eqref{Y-useful-assumption} follows from the simpler condition
\begin{align}\label{t-condition} 
 t-t^\star \geq  ( \bar{\epsilon}^{-(m-1)} +1)  \| Y \|_{[t^\star,1]}^{-(m-1)}     .
\end{align}

\noindent { \bf Step 3: Coming down from infinity for bounded coefficients.}
We continue working under the boundedness Assumption~\ref{Ass-Sigma-Bounded} and aim to establish 
the estimate \eqref{e:thm-bound}.
We define a finite sequence $0 = t_0 < t_1<t_2 < \ldots < t_N =1$ by setting
\begin{align}\label{e:def-t-increment}
t_{n+1} - t_n = A \| Y \|_{[t_n,1]}^{-(m-1)},
\end{align}
where $A =  \bar{\epsilon}^{-(m-1)} + 1$ (see  \eqref{t-condition}), 
as long as the resulting $t_{n+1}$ satisfies $t_{n+1} < 1$. As soon as $t_{n+1}$ defined by this rule is $\geq 1$ we set 
$t_{n+1} = t_N =1$ and terminate the sequence. Note that the sequence $t_{n+1} - t_n$ is increasing by definition 
(at least before passing $1$),  so that this algorithm must terminate after a finite number of steps.

The time $t$ is contained in precisely one interval $(t_n,t_{n+1}]$ and precisely one of the following must hold:
\begin{enumerate}
\item either there exists a $k \leq n$ and  a tree $h$ such that 
\begin{equation*}%\label{eq:simplifying_assump-neg}
\epsilon \|Y\|_{[t_k, 1]}^{|h| m   \alpha} <  [\X : h] \;,
 \end{equation*} 
 \item or for all $k \leq n$ all trees $h$ satisfy condition \eqref{eq:simplifying_assump} for $\delta = m$.
\end{enumerate}
In case (1), the  desired estimate \eqref{e:thm-bound} follows immediately from 
\[
|Y(t)| \leq   \|Y\|_{[t_k, 1]} \leq  \frac{1}{\epsilon}   [\X : h]^{\frac{1}{|h| m \alpha}}. 
\]
In case (2)  invoking \eqref{thereweare}  repeatedly for $t_* = t_k$ yields 
\begin{equation}\label{e:useful-estimate}
\| Y \|_{[t_n,1] } \leq 2^{k-n} \| Y \|_{[t_k,1]} \qquad \text{ for } k =0, 1, 2 , \ldots n. 
\end{equation}
We split (2) further into the following three cases:
\begin{enumerate}
\item[(2a)] $n +1 < N $. 
 \item[(2b)] $n+1=N$ and $t_n \leq \frac12$.
 \item[(2c)] $n+1=N$ and $t_n > \frac12$.
\end{enumerate}
In case (2a) we write using \eqref{e:useful-estimate}
\begin{align*}
t_{n+1} &= \sum_{k=0}^n   t_{k+1} - t_k = A  \sum_{k=0}^n   \| Y \|_{[t_k,1]}^{-(m-1)}  \leq  A  \| Y \|_{[t_n,1]}^{-(m-1)}  \sum_{k=0}^n   2^{- (m-1) (n-k)} \\
&\lesssim  \| Y \|_{[t_n,1]}^{-(m-1)} ,
\end{align*}
which implies 
\begin{equation*}
|Y(t)| \leq  \| Y \|_{[t_n,1]} \lesssim t_{n+1}^{-\frac{1}{m-1}} \leq t^{-\frac{1}{m-1}},
\end{equation*}
as desired. 

For (2b) to hold by definition \eqref{e:def-t-increment} we have
\[
A \| Y \|_{[t_n,1]}^{-(m-1)} \geq \frac12 ,
\]
and in particular 
\[
|Y(t)| \leq \| Y \|_{[t_n,1] } \leq  \Big(\frac{1}{2A}\Big)^{-\frac{1}{m-1}},
\]
which implies \eqref{e:thm-bound}. 

Finally, in case (2c) we can apply case (2a) to $t_n$ to get
\[
|Y(t)| \leq \|Y\|_{[t_n,1]} \leq t_n^{-\frac{1}{m-1}}  \leq \Big( \frac{1}{2} \Big)^{-\frac{1}{m-1}}, 
\]
completing the proof of \eqref{e:thm-bound}.
\end{proof}

\begin{proof}[Proof of Theorem~\ref{thm:main} (2) - The case of polynomial coefficients.] 
The proof follows the same general lines as the proof in the case of bounded coefficients and we only explain the differences.
Combining, their definition \eqref{e:DefUf} with the estimate on $\Upsilon$ provided by Lemma~\ref{l:Assumption-imply-Upsilon}, 
yields  
\begin{equation}\label{e-Uestimate} 
\begin{split}
U_{\mu}(f, \Y,I )& \lesssim \sup \big\{
\jbrac{ Y_s + a}^{  (\gamma -1)  (|f|+1) } : 
s,t \in I, |a| \le E^{\Y}_{s,t} + |Y_{t}-Y_{s}| \big\}  \; ,\\
U_{\mu}(f,\bar{f}, \Y,I) &\lesssim   
\sup
\bigg\{
\jbrac{Y_{s} + z}^
 { (\gamma -1) (|f|+1)  +1 - \num \bar{f} }
   \jbrac{Y_{s}}^
   {(\gamma-1) |\bar{f} |  + \num \bar{f} } : 
\begin{array}{c}
s, t \in I, \\
|z| \le E^{\Y}_{s,t}
\end{array}
\bigg\} \; . 
\end{split}
\end{equation}
We follow the same approach as Step 1 of the proof in the case of bounded coefficients, fixing $t^\star \in [0,1)$, but here  
the condition   \eqref{e:step4-condition2} is  replaced by
 \begin{align}
\label{e:step4-condition2-step7}
L \leq   \epsilon_2   \jbrac{Y}_{[t^\star,1]}^{-(m-1)}
\end{align} 
where we have introduced an auxiliary parameter $\epsilon_2 \leq 1$ that will ultimately be chosen small but still large relative to $\epsilon$ and we have changed values of exponents from $\delta = m$ to $\delta = m-1$. 
Instead of \eqref{eq:simplifying_assump} we assume that for $h \in \rtree{N}$ 
\begin{equation}\label{eq:simplifying_assump-Step7}
[\X : h]
\le     \epsilon_1 \|Y\|_{[t^\star, 1]} ^{ |h|  (\alpha (m-1)    - (\gamma -1)) } \;.
\end{equation}
Note that the exponent $|h|  (\alpha (m-1)    - (\gamma -1))$ appearing on the right hand side is positive if and only if the upper bound $\gamma < (m-1) \alpha +1$ in Assumption~\ref{gamma-condition-m} holds. 
Under \eqref{e:step4-condition2-step7} and \eqref{eq:simplifying_assump-Step7} we have 
\begin{align}
\notag
 E^{\Y}_{s,t} = &\sum_{h \in \rtree{N-1}} |\scal{\X_{s,t}}{h} \Upsilon[h](Y_{s})| \\
\notag
& \lesssim \sum_{h \in \rtree{N-1}}   \Big( \epsilon_2  \jbrac{ Y}_{[t^\star,1]}^{-(m-1)} \Big)^{\alpha |h|} \epsilon_1 \|Y\|_{[t^\star, 1]} ^{ |h|  (\alpha (m-1)    - (\gamma -1)) } \jbrac{Y}_{[t^\star, 1]} ^{(\gamma -1) |h| +1} \\
\label{e:error-in-evaluation-bound}
& \lesssim \epsilon_2^{\alpha} \epsilon_1 \jbrac{Y}_{[t^\star, 1]} \;.
\end{align}
Note that in the first inequality, we have made use of \eqref{upsilon-bound-sigma-poly} and the assumption $\gamma \geq 1$ in \eqref{gamma-condition-m} which ensures that the exponent $(\gamma -1) |h| +1$ in the last term is non-negative
and permits to use the estimate
\begin{align}\label{an-example}
\Big| \Upsilon[h](Y_{s}) \Big| \lesssim  \jbrac{Y_s}^{(\gamma - 1) |h| +1} \leq  \jbrac{Y}_{[t^\star, 1]} ^{(\gamma -1) |h| +1} \;.
\end{align}
We now enforce that 
\begin{align}\label{choiceetac}
\epsilon_1 \epsilon_2^\alpha  \leq \frac{ \epsilon}{C} \;,
\end{align}
Using this,  \eqref{e:error-in-evaluation-bound}, and once more the fact that $(\gamma - 1) \geq 1$ the estimate \eqref{e-Uestimate} simplifies to 
\begin{align}
\notag
U_{\mu}(f, \Y,I )& \lesssim 
 \jbrac{ Y}_{[t^\star, 1]}^{  (\gamma -1)  (|f|+1) }  ,\\
U_{\mu}(f,\bar{f}, \Y,I) &\lesssim   
\jbrac{Y}_{[t^\star, 1]}^
 { (\gamma -1) (|f|+\bar{f} +1)  +1 }
\label{e-Uestimate-simplified}   \;.
\end{align}
Note that these estimates in conjunction with \eqref{choiceetac}, \eqref{e:step4-condition2-step7}, and Assumption \eqref{eq:simplifying_assump-Step7} imply \eqref{e:step2-condition} 
and in this case \eqref{e:step3-final1} and \eqref{e:step3-final2} turn into 
\begin{align}
\notag
L^\alpha[Z]_{\alpha,I} 
 &\lesssim	
	\sum_{f \in\forest{N-1}  }\sum_{\mu\in[\dimX]}
		[\X:[f]_\mu]  
			\\
\notag
	& \quad   \cdot \Bigg \{
		 	 L^{ (|f| +1) \alpha +1 }   
			\jbrac{ Y }_{[t^\star, 1]}^{ (|f|+1) (\gamma -1) }
			 \| Y \|_I^m		
			+
			\max_{\bar{f}  \in \partial \forest{N- |f| -1}  }  
 				L^{ (|\bar{f} | +|f| +1 ) \alpha  } 
 				\jbrac{Y}_{[t^\star, 1]}^{(|\bar{f}| + |f|+1) (\gamma -1) +1}
				[\X : \bar{f} ]
		 \Bigg\} 	\\		
\notag
&
+ \sum_{h \in \rtree{N}} 
	L^{|h|\alpha} [\X:h] \jbrac{Y}_{[t^\star, 1]}^{|h| (\gamma -1) +1}   \\
%\notag
%L^\alpha[Z]_{\alpha,I} 
& \lesssim 
\epsilon ( L   \jbrac{Y}_{[t^\star,1]}^m + \jbrac{Y}_{[t^\star,1]}  + \jbrac{Y}_{[t^\star,1]}  ) 
\label{e:step7-Zbound}
 \lesssim  \epsilon \jbrac{Y}_{[t^\star,1]} \;,
\end{align}
and 
\begin{align*}
%\label{e:step7-Ybound}
L^\alpha[Y]_{\alpha,I}  
&\lesssim 
\epsilon \jbrac{Y}_{[t^\star,1]}   	+	 L   \|Y \|_{[t^\star,1]}^m  \lesssim  \epsilon_2  \jbrac{Y}_{[t^\star,1]}^m \;.
\end{align*}
As above in Step 2 we next apply Lemma~\ref{lemma:MW} on the interval $[t_\star,1]$ yielding  
\begin{align}
\notag
\jbrac{Y(t)}  
  \lesssim 
   \max \Big\{ 
	 & (t-t^\star -L)^{-\frac{1}{m-1}}; 
	  \Big( \frac{\epsilon}{\epsilon_2} \Big)^{\frac{1}{m}}   \|Y \|_{[t^\star,1]};  \\
%\notag
& \qquad 	
	\Big(     \|Y \|_{[t^\star,1]}^{m-1} \Big( \epsilon_2  \jbrac{Y}_{[t^\star,1]} \Big)   \Big)^{\frac{1}{m}} ; 
	  \epsilon_2  \jbrac{ Y}_{[t^\star,1]}  
     \Big\}.
     \notag
%\label{Step7-1} 
\end{align}

Then choosing first $\epsilon_2$ small to control the pre-factors in the second line  and then also  $\epsilon$ small to also make the pre-factors in the first line small (this amounts to choosing $\epsilon_1$ in \eqref{eq:simplifying_assump-Step7} small) 
we obtain
\begin{align}
\jbrac{Y(t)}  
  \leq  
   \max \Big\{ 
	 & C (t-t^\star -L)^{-\frac{1}{m-1}};  \frac{1}{2} \jbrac{Y}_{[t^\star,1]}
	        \Big\}.
	        \notag
%\label{Step7-2}
\end{align} 
The proof  concludes as in the case of bounded coefficients, Step 3  above, where in \eqref{e:def-t-increment} and in what follows, the role of $\| Y \|_{[t_k,1]}$ is played by $\jbrac{Y}_{[t_{k},1]}$. 
\end{proof}

\begin{proof}
[Proof of Theorem~\ref{thm:main2}]
We first show \eqref{e:small-time-bound1}  under Assumption~\ref{Ass-Sigma-Bounded} (i.e. for bounded coefficients). We first define the auxiliary time  
\begin{align}\label{def-T0}
T_0 = \inf \big\{t \in [0,1] \colon \jbrac{y(t)} \geq 3 \jbrac{y_0}  \big \}    \;,
\end{align}
with the convention that $T_0 = 1$ if the set is empty. Continuity of $|y(t)|$ in $t$ implies that if $T_0 <1$ we have 
\begin{align}\label{observationTbar}
\jbrac{y(T_0)} = 3 \jbrac{ y_0} \;.
\end{align}
Furthermore, let $T_1$ and $T_2$ be as in \eqref{e:def-T1T2-1}, where $\epsilon_2 $ is assumed to be small enough to play the role of $\epsilon$ in the absorption condition
  \eqref{e:step4-condition1} above and $\epsilon_1$ is fixed below, and set 
  \[
  T^{\star} = \min \{ T_0, T_1, T_2 \} \;.
  \]
Note that by choice of $T_2$ condition \eqref{e:step4-condition1} is automatically satisfied for $I = [0, T^\star]$  and \eqref{e:step4-1} turns into  
\begin{align*}
(T^{\star} )^\alpha [Z]_{\alpha, [0, T^\star] }  \lesssim \epsilon_2 +  \epsilon_2 \epsilon_1  3^m \jbrac{y_{0}} \;, 
\end{align*}
where we have used the fact that by definition 
\[
T^\star \jbrac{ Y}_{[0, T^\star ]}^m  \leq  \epsilon_1  \frac{1}{\jbrac{y_0}^{m-1}}   3^m  \jbrac{y_0}^m  \leq   \epsilon_1   3^m  \jbrac{y_0}    \;.
\] 
Hence we get from \eqref{RDE3}, choosing $\epsilon_1 = \frac{1}{2} 3^{-m}$ and possibly making $\epsilon_2$ smaller
\begin{align}
\notag
\jbrac{Y}_{[0, T^{\star} ]}
\le \| Y \|_{[0, T^{\star} ]} + 1 & \leq |y_0 | +  T^{\star}  \| Y \|_{[0, T^\star ]}^m  +  (T^{\star} )^\alpha [Z]_\alpha \\
\notag
&  \leq |y_0 |    + \epsilon_1  3^m  \jbrac{y_0} + C\Big(   \epsilon_2 +  \epsilon_2 \epsilon_1  3^m  \jbrac{y_0} \Big)   \\
\label{mal-wieder-conclusion}
&\leq  \frac{3}{2}  \jbrac{y_0}  + 1 \leq \frac{5}{2}   \jbrac{y_0}  \;. 
\end{align}
To conclude that \eqref{e:small-time-bound1} holds indeed, it remains to see that $  T^{\star} = \min \{ T_0, T_1, T_2 \}  = \min \{  T_1, T_2 \} $. This is indeed the case, because otherwise we would 
have $\| Y \|_{[0, T^\star ]} = \| Y \|_{[0, T_0 ]}$ and $T_0 < 1$, but then \eqref{mal-wieder-conclusion} would  contradict \eqref{observationTbar}.

We next pass to the case of polynomially bounded $\sigma$, i.e. to establishing \eqref{e:small-time-bound3} under Assumption~\ref{Ass-Sigma-Polynomial}. Let $T_0$ be as in \eqref{def-T0} and 
set 
  \[
  T^{\star} = \min \{ T_0, T_1, T_2 \} \;,
  \]
where this time $T_2$ is defined in \eqref{e:def-T1T2-bis}. We claim that on $I = [0,T^\star]$ satisfies the Assumption~\eqref{e:step2-condition}. Indeed, just as in \eqref{e:error-in-evaluation-bound} above one can see that for $s,t$ we have
\begin{align*}
 E_{s,t}^{\Y}  = 
 &\sum_{h \in \rtree{N-1}} |\scal{\X_{s,t}}{h} \Upsilon[h](Y_{s})|\\
  & \lesssim  \sum_{h \in \rtree{N-1}} (T^{\star})^{|h| \alpha} [\X : h ] \| \Upsilon[h](Y_{\bullet}) \|_{[0, T^\star]}   \\
&  \lesssim  \sum_{h \in \rtree{N-1}} \epsilon_2^{|h| \alpha}   \frac{1}{[\X \colon h  ]}   \Big( \frac{1}{\jbrac{y_0}}  \Big)^{(\gamma-1)|h|}       [\X : h ]  \jbrac{Y}_{[0, T^\star]}^{(\gamma -1) |h| +1}   \\
& \lesssim  \sum_{h \in \rtree{N-1}} \epsilon_2^{|h| \alpha}   \frac{1}{[\X \colon h  ]}   \Big( \frac{1}{\jbrac{y_0}}  \Big)^{(\gamma-1)|h|}       [\X : h ] \jbrac{Y}_{[0, T^\star]}^{(\gamma -1) |h| +1} \\
&\lesssim  \epsilon_2^{\alpha}  \jbrac{Y}_{[0, T^\star]} \;,
\end{align*}
where in the first inequality the fact $T^\star \leq T_2$ and in the last inequality $T^\star \leq T_0$ was used. From this one can conclude just as in \eqref{e-Uestimate-simplified} that 
\begin{align}
\notag
U_{\mu}(f, \Y,[0, T^\star] )& \lesssim 
 \jbrac{ Y}_{[0, T^\star]}^{  (\gamma -1)  (|f|+1) }  ,\\
U_{\mu}(f,\bar{f}, \Y,[0, T^\star]) &\lesssim   
\jbrac{Y}_{[0, T^\star]}^
 { (\gamma -1) (|f|+\bar{f} +1)  +1 }\;,
\notag
%\label{e-Uestimate-simplified-bis}   \;,
\end{align}
which in turn leads to 
\begin{align*}
L^\alpha[Z]_{\alpha,[0, T^\star]}  \lesssim \epsilon \jbrac{Y}_{[0, T^\star]} \lesssim  \epsilon \jbrac{y_0}  \,
\end{align*}
where the first inequality follows like \eqref{e:step7-Zbound} and the second follows from $T^\star \leq T_0$. The rest of the argument follows as in the case of bounded coefficients.
\end{proof}

\begin{proof}[Proof of Corollary~\ref{coro} ]
We treat the case of bounded coefficients first. 
If we have $t \geq T_1 =  \epsilon_1 \frac{1}{\jbrac{y_0}^{m-1}}$ or if $t \geq T_2 = \epsilon_2 \min   \bigg\{  \frac{1}{ [\X \colon h  ]^{\frac{1}{|h| \alpha }}: h \in \rtree{N}  }  \bigg\}$  defined in \eqref{e:def-T1T2-1}
 we use the estimate \eqref{e:thm-bound}  
\[
|Y(t) |  \lesssim  \max \Big\{   t^{-\frac{1}{m-1}}   \; ,\max_{h \in \rtree{N}} [\X:h]^{\frac{1}{ m\alpha|h|}}  \Big\}  \;. 
\]
In the second term on the right-hand-side, we can replace the exponent $\frac{1}{ m\alpha|h|}$ by the slightly larger $\frac{1}{ (m-1)\alpha|h|}$ at the expense of adding the additional $1$ into the maximum in \eqref{e:cor-bound}. 
For the first term, we use the lower bound on $t$ provided by $T_1$ and $T_2$. Indeed, for $t \geq T_1$ we get 
\begin{align}\label{larget1}
t^{-\frac{1}{m-1}}  \leq \epsilon_1^{-\frac{1}{m-1}} \jbrac{ y_0} \;
\end{align}
and for $t \geq T_2$ we have 
\[
t^{-\frac{1}{m-1}}  \leq \epsilon_2^{-\frac{1}{m-1}}  \max_{h \in \rtree{N}}   [\X \colon h  ]^{\frac{1}{(m-1) |h| \alpha }}  \;
\]
in each case establishing \eqref{e:cor-bound}. Finally, if $t \leq \min\{ T_1, T_2 \} $  \eqref{e:cor-bound} follows from \eqref{e:small-time-bound1}. 

The argument for polynomial coefficients is very similar. If  $t \geq T_1$ or if $t \geq T_2$, where this time $T_2$ is defined by \eqref{e:def-T1T2-bis}, we use 
the estimate  \eqref{e:thm-bound-bis}
\begin{align*}%\label{e:thm-bound-bisbis}
|Y(t) | \lesssim  \max \Big\{  t^{-\frac{1}{m-1}}  \;, \max_{h \in \rtree{N}} [\X:h]^{\frac{1}{ ((m-1) \alpha- \gamma+1)|h|}} \Big\} \;.
\end{align*}
For $t \geq T_1$ \eqref{larget1} gives the desired \eqref{e:cor-bound-bis}, while for $t \geq T_2$ we write
\begin{align*}
t^{-\frac{1}{m-1}}  &\leq \epsilon_2^{-\frac{1}{m-1}}   \jbrac{y_0}^{\frac{\gamma-1}{(m-1) \alpha}}   \max_{h \in \rtree{N}}   [\X \colon h  ]^{\frac{1}{(m-1) |h| \alpha }}  \\
&\lesssim \jbrac{y_0} +  \max_{h \in \rtree{N}} [\X:h]^{\frac{1}{ ((m-1) \alpha- \gamma+1)|h|}}  \;,
\end{align*}
where in the second inequality we have used the estimate $ab \leq a^p + b^q$ for $p = \frac{(m-1)\alpha}{ \gamma -1}$ and $q = \frac{(m-1)\alpha}{ (m-1)\alpha  - \gamma +1 }$.
 As above, if $t \leq \min\{ T_1, T_2 \} $  \eqref{e:cor-bound-bis} follows from \eqref{e:small-time-bound3}. 
\end{proof}

\bibliographystyle{alpha}
\bibliography{biblio_stage_M1}
\end{document}